\documentclass[a4paper,11pt]{amsart}
\usepackage{amstext,amsfonts,amsthm,graphicx,amssymb,amscd,epsfig}
\usepackage{amsmath}
\usepackage[ansinew]{inputenc}    
\usepackage{psfrag}
\usepackage{mathrsfs}
\usepackage{xypic}
\usepackage{graphicx}
\usepackage[all,knot,arc,import,poly]{xy}
\usepackage{color}
\usepackage{comment}
\usepackage{subfigure}

\theoremstyle{definition}
\newtheorem{definition}{Definition}[section]
\newtheorem{ex}[definition]{Example}
\newtheorem{rem}[definition]{Remark}

\theoremstyle{plain}
\newtheorem{prop}[definition]{Proposition}
\newtheorem{lem}[definition]{Lemma}
\newtheorem{coro}[definition]{Corollary}
\newtheorem{teo}[definition]{Theorem}

\newfont{\bbb}{msbm10 scaled\magstephalf}     

\def\reg{\operatorname{reg}}
\def\sing{\operatorname{sing}}
\def\Hess{\operatorname{Hess}}
\def\rank{\operatorname{rank}}

\title{Axial curvatures for corank 1 singular $n$-manifolds in $\mathbb R^{n+k}$}

\author{P. Benedini Riul, J. L. Deolindo-Silva, R. Oset Sinha}

\date{}

\address{Departamento de Estat\'istica, F\'isica e Matem\'atica, 
Universidade Federal de S\~{a}o Jo\~{a}o Del-Rei - UFSJ, Campus Alto Paraopeba, 36497-899, Ouro Branco, Brazil}

\email{benedini@ufsj.edu.br}

\address{Departamento de Matem\'atica.
Universidade Federal de Santa Catarina - UFSC, 89036-004 -
Blumenau-SC,  Brazil}

\email{jorge.deolindo@ufsc.br}

\address{Departament de Matem\`{a}tiques,
Universitat de Val\`encia, Campus de Burjassot, 46100 Burjassot,
Spain}

\email{raul.oset@uv.es}

\thanks{Work of J. L. Deolindo-Silva partially supported by FAPESP Grant number 2019/07316-0.}
\thanks{Work of R. Oset Sinha partially supported by Grant PGC2018-094889-B-100 funded by MCIN/AEI/ 10.13039/501100011033 and by ``ERDF A way of making Europe"}

\subjclass[2000]{Primary 57R45; Secondary 53A07, 58K05} \keywords{axial curvature, singular manifolds in Euclidean spaces, second order geometry, umbilic curvature, frontals}

\begin{document}

\begin{abstract}
For singular $n$-manifolds in $\mathbb R^{n+k}$ with a corank 1 singular point at $p\in M^n_{\sing}$ we define up to $l(n-1)$ different axial curvatures at $p$, where $l=\min\{n,k+1\}$. These curvatures are obtained using the curvature locus (the image by the second fundamental form of the unitary tangent vectors) and are therefore second order invariants. In fact, in the case $n=2$ they generalise all second order curvatures which have been defined for frontal type surfaces. We relate these curvatures with the principal curvatures in certain normal directions of an associated regular $(n-1)$-manifold contained in $M^n_{\sing}$. We obtain many interesting geometrical interpretations in the cases $n=2,3$. For instance, for frontal type 3-manifolds with 2-dimensional singular set, the Gaussian curvature of the singular set can be expressed in terms of the axial curvatures. Similarly for the curvature of the singular set when it is 1-dimensional. Finally, we show that all the umbilic curvatures which have been defined for singular manifolds up to now can be seen as the absolute value of one of our axial curvatures.  
\end{abstract}

\maketitle

\section{Introduction}

In the last 15 years the study of the differential geometry of singular surfaces has flourished to be an area of great interest for researchers from many different backgrounds. These objects are cherished by differential geometers as much as by singularists. Even contact topologists are encountering singular objects when studying wave fronts and frontals, and Gauss-Bonnet type theorems provide a link to the geometry. Singularity Theory has proved to be the ideal framework to study these objects and its approach for regular surfaces (see the recent book \cite{Livro}) can be adapted for the singular case. Ways of studying the geometry of the singular surface are relating it to the geometry of a regular surface from which it is obtained by an orthogonal projection (\cite{BenediniOset2,BallesterosTari,OsetSinhaTari}), or studying its contact with planes and spheres (\cite{Oset/Saji,OsetSinhaTari2}). Another way is to define curvatures which give information about the surface. For example, in \cite{SUY} a singular and a limiting normal curvature were defined for certain frontal type singularities and a Gauss-Bonnet theorem was proved using the former. In \cite{HHNUY} the intrinsity of this kind of invariants is studied and in \cite{teramoto} the limiting normal curvature is interpreted as a principal curvature. 

In \cite{MartinsBallesteros} the authors defined the curvature parabola at a singular corank 1 point in a surface in $\mathbb R^3$. This is the image by the second fundamental form of the unitary tangent vectors and plays an analogous role to the curvature ellipse defined by Little in \cite{Little}. Using this parabola they defined an umbilic curvature which captures the round geometry of the surface and generalises the limiting normal curvature defined for cuspidal edges in \cite{MartinsSaji}. In \cite{OsetSaji}, the third author and K. Saji, using the curvature parabola, defined for any corank 1 singular surface an axial curvature which generalises the singular curvature defined in \cite{SUY}. The axial curvature was defined using the properties of the parabola, so it was not clear until now how to generalise this to higher dimensions.

For $p\in M^n_{\sing}$ a corank 1 singular point in an $n$-manifold in $\mathbb R^{n+k}$, the curvature locus can have many different topological types and can even have singularities. These types of loci have been studied in \cite{Benedini/Sinha/Ruas} for $n=2$ and $k=2$ and in \cite{BenediniOset2,Benedini/Sinha/Ruas2,BenediniRuasSacramento} for $n=3$ and $k=2$.

In this paper, using the curvature locus, we define axial curvatures for any corank 1 singular $n$-manifold in $\mathbb R^{n+k}$ which can be seen as principal curvatures. These curvatures generalise the umbilic and axial curvatures for surfaces in $\mathbb R^3$. We define a special adapted frame of axial vectors $\{v_a^1,\ldots,v_a^l\}$, where $l=\min\{n,k+1\}$, in the normal space $N_pM$ and for each axial vector we define the axial curvatures as the critical values of the projection of the curvature locus onto the direction of the corresponding axial vector.

In Section \ref{defs}, we give the main definitions and show that there can be up to $l(n-1)$ different axial curvatures at a point $p\in M^n_{\sing}$. Section \ref{sups} is devoted to the particular case of surfaces in $\mathbb R^n$. We give formulas for the axial curvatures and give geometrical interpretations for them. For example, we show that for certain surfaces there is a distinguished curve on them with curvature $\kappa$ which satisfies $\kappa^2=(\kappa_{a_1})^2+(\kappa_{a_2})^2$, where $\kappa_{a_1}$ and $\kappa_{a_2}$ are the primary and secondary axial curvatures which generalise the axial and umbilic curvatures from \cite{OsetSaji} and \cite{MartinsBallesteros}, respectively. In Section \ref{M3} we study 3-manifolds in $\mathbb R^{3+k}$. For $k=1,2$ we define adapted frames and prove the following elegant relation: For a corank 1 singular manifold $M^{n}_{\sing}\subset\mathbb R^{n+k}$, it is possible to take a Monge form
\begin{equation}\label{Monge-final}
 f(x_1,\ldots,x_n)=(x_1,\ldots,x_{n-1},f_{1}(x_1,\ldots,x_n),\ldots,f_{k+1}(x_1,\ldots,x_n)),  
\end{equation}
where $\frac{\partial f_{\ell}}{\partial x_i}=0$ for $\ell=1,\ldots,k+1$ and $i=1,\ldots,n$. We show that the axial curvatures corresponding to the axial vector $v_a^i$ coincide with the $v_a^i$-principal curvatures of the regular $(n-1)$-manifold $M_{\reg}^{n-1}$ given by $f(x_1,\ldots,x_{n-1},0)$. In this sense, the axial curvatures can be understood as principal curvatures of singular manifolds. Using this relation we obtain interesting geometrical interpretations. For instance, for frontal type 3-manifolds with a smooth 2-dimensional singular set, the Gaussian curvature of the singular set can be expressed in terms of the axial curvatures. The same can be done for the curvature of the singular set when it is 1-dimensional. This opens countless directions in which to study the differential geometry of higher dimensional frontals. Finally, in Section \ref{umbs} we give an overview of all the different umbilic curvatures which have been defined up to now (both in the regular and singular setting), which are related to the umbilical focus of centers of spheres with degenerate contact with the manifold, and prove that all of them can be obtained as the absolute value of an axial curvature.


\section{The geometry of singular $n$-manifolds in $\mathbb R^{n+k}$}\label{prelim}

In this section we review the basic definitions and results related to the second order geometry of corank 1 singular $n$-manifolds in $\mathbb R^{n+k}$. For more details see \cite{BenediniOset2,Benedini/Sinha/Ruas,BenediniRuasSacramento,MartinsBallesteros}.

Let $M^n_{\sing}\subset \mathbb R^{n+k}$ be a $n$-manifold with a singularity of corank 1. We can consider $M^n_{\sing}$ as the image of a smooth map $g : \tilde{M}\to\mathbb R^{n+k}$ from a smooth regular $n$-manifold $\tilde{M}$ whose differential map has rank $\geq n-1$ at any point such that $g(q)=p$. 
Consider $\phi:U\rightarrow\mathbb{R}^{n}$ a local coordinate system defined in an open neighborhood $U$ of $q$ at $\tilde{M}$. Using this construction, we may consider a local parametrisation $f=g\circ\phi^{-1}$ of $M^{n}_{\sing}$ at $p$ (see the diagram below).
$$
\xymatrix{
\mathbb{R}^{n}\ar@/_0.7cm/[rr]^-{f} & U\subset\tilde{M}\ar[r]^-{g}\ar[l]_-{\phi} & M^{n}_{\sing}\subset\mathbb{R}^{n+k}
}
$$
Considering $dg_{q}:T_{q}\tilde{M}\rightarrow T_{p}\mathbb{R}^{n+k}$ the differential map of $g$ at $q$, the {\it tangent space}, $T_{p}M^{n}_{\sing}$, at $p$ is given by $\mbox{Im}\ dg_{q}$ that degenerates to a $(n-1)$-space and the {\it normal space} at $p$, $N_{p}M^{n}_{\sing}$, is the $k+1$-space of orthogonal directions to $T_p M^n_{\sing}$ in $\mathbb R^{n+k}$ such that  $T_{p}M^{n}_{\sing}\oplus N_{p}M^{n}_{\sing}=T_{p}\mathbb{R}^{n+k}$. 


 The {\it first fundamental form} of $M^{n}_{\sing}$ at $p$, $I:T_{q}\tilde{M}\times T_{q}\tilde{M}\rightarrow \mathbb{R}$ is given by
$I(u,v)=\langle dg_{q}(u),dg_{q}(v)\rangle$, $\forall\ u,v\in T_{q}\tilde{M}$. 
This induces a pseudometric in $T_q\mathbb R^n$ since the image of non-zero vectors can be zero. The {\it second fundamental form} of $M^{n}_{\sing}$ at $p$, $II:T_{q}\tilde{M}\times T_{q}\tilde{M}\rightarrow N_{p}M^{n}_{\sing}$  is given by
$II(u,v)=\pi_2(d^2f_{\phi(q)}(d\phi(u,v))),$
where  $\pi_2:T_{p}\mathbb{R}^{n+k}\rightarrow N_{p}M^{n}_{\sing}$ is the orthogonal projection and is extended to the whole space uniquely as a symmetric bilinear map.

Given a normal vector $\nu\in N_{p}M^{n}_{\sing}$, the {\it second fundamental form along $\nu$}, $II_{\nu}:T_{q}\tilde{M}\times T_{q}\tilde{M}\rightarrow\mathbb{R}$ is given by $II_{\nu}(u,v)=\langle II(u,v),\nu\rangle$, for all $u,v\in T_{q}\tilde{M}$. 

Let $C_{q}= \{u\in T_q\tilde{M} \;|\; I(u,u)= 1\}$ be the subset of unit tangent vectors and let $\eta:C_{q}\rightarrow N_{p}M^{n}_{\sing}$ be the map given by $\eta(u)=II(u,u)$. The \emph{curvature locus} of $M^{n}_{\sing}$ at $p$, denoted by $\Delta_{p}$, is the subset $\eta(C_q)$. The curvature locus does not depend on the choice of the local coordinates
of $\tilde{M}$. Define $Aff_p$ as the {\it affine space} of minimal dimension which contains $\Delta_p$.


\subsection{Singular surfaces in $\mathbb R^{2+k}$, $k\geq1$}\label{singular}

Let $M^{2}_{\sing}$ be a corank $1$ surface at $p$.  
If $\{\partial_x,\partial_y\}$ is a basis for $T_{q}\tilde{M}$ and using the parametrisation $f$, the coefficients of the first fundamental form are:
$E(q) = I(\partial_x,\partial_x) =\langle f_x,f_x\rangle(q)$, $F(q) = I(\partial_x,\partial_y) =\langle f_x,f_y\rangle(q)$ and $G(q) = I(\partial_y,\partial_y) =\langle f_y,f_y\rangle(q)$
and taking $u = a\partial_x + b\partial_y\in T_q\mathbb R^2$, 
$$
I(u,u) = a^2E(q) + 2abF(q) + b^2G(q).
$$
The second fundamental form of $M^{2}_{\sing}$ at $p$, $II:T_{q}\tilde{M}\times T_{q}\tilde{M}\rightarrow N_{p}M^{2}_{\sing}$  is given by
$$
\begin{array}{c}
     II(\partial_{x},\partial_{x})=\pi_2(f_{xx}(\phi(q))),\  II(\partial_{x},\partial_{y})=\pi_2(f_{xy}(\phi(q))),\
      II(\partial_{y},\partial_{y})=\pi_2(f_{yy}(\phi(q))).
\end{array}
$$

The coefficients of $II_{\nu}$ with respect to the basis $\{\partial_{x},\partial_{y}\}$ of $T_{q}\tilde{M}$ are
$$
\begin{array}{cc}
     l_{\nu}(q)=\langle \pi_2(f_{xx}),\nu\rangle(\phi(q)),\ m_{\nu}(q)=\langle \pi_2(f_{xy}),\nu\rangle(\phi(q)),  \\
     n_{\nu}(q)=\langle \pi_2(f_{yy}),\nu\rangle(\phi(q)).
\end{array}
$$
Thus, if $u=\alpha\partial_{x}+\beta\partial_{y}\in T_{q}\tilde{M}$ and
fixing an orthonormal frame $\left\{\nu_{1},\ldots,\nu_{k+1}\right\}$ of $N_{p}M^{2}_{\sing}$, the second fundamental form is given by
$$
\begin{array}{c}\label{eq.2ff}
II(u,u) =\displaystyle\sum_{i=1}^{k+1}II_{\nu_{i}}(u,u)\nu_{i}=\sum_{i=1}^{k+1}(\alpha^{2}l_{\nu_{i}}(q)+2\alpha\beta m_{\nu_{i}}(q)+\beta^{2}n_{\nu_{i}}(q))\nu_{i}. \\
\end{array}
$$

It is possible to take a coordinate system $\phi$ and make rotations in the target in order to obtain a local parametrisation in the Monge form
as in (\ref{Monge-final}).

Taking an orthonormal frame $\{\nu_{1},\ldots,\nu_{k+1}\}$ of $N_{p}M^2_{\sing}$, the curvature locus $\Delta_{p}$ can be parametrised by
$$
\eta(y)=\sum_{i=1}^{k+1}(l_{\nu_{i}}+2m_{\nu_{i}}y+n_{\nu_{i}}y^{2})\nu_{i},
$$
where 
each parameter $y\in\mathbb{R}$ corresponds to a unit tangent direction $u=\pm\partial_{x}+y\partial_{y}=(\pm1,y)\in C_{q}$. We denote by $y_{\infty}$ the parameter corresponding to the tangent direction given by $u=\partial_{y}=(0,1)$.

In the case of $k=1$, the curvature locus $\Delta_p$ is a planar parabola that may degenerate into a line, a half-line or a point. 

When $\Delta_p$ degenerates to a line, a half-line or a point, a special adapted frame $\{\nu_2, \nu_3\}$ of $N_pM^{2}_{\sing}$ was defined in \cite{MartinsBallesteros}, and this definition was extended for the case when $\Delta_p$ is a non-degenerate parabola in \cite{OsetSaji}, where $\nu_2$ was called the {\it axial vector}, $v_a$. With this frame and $u\in C_q$, $II(u, u) = II_{v_a} (u, u)\nu_2 +II_{\nu_3} (u, u)\nu_3$. When $\Delta_p$ degenerates to a line, a half-line or a point $II_{\nu_3} (u, u)$ does not depend on $u$ up to sign and the {\it umbilic curvature} of $M^{2}_{\sing}$ at $p$ is defined in \cite{MartinsBallesteros} by
$\kappa_u =|\langle II(u,u),\nu_3\rangle|=|II_{\nu_3}(u,u)|.$ On the other hand the {\it axial curvature} is defined in \cite{OsetSaji} as $\kappa_a(p) = min\{K_{v_a}(u) : u\in C_q\} = min\{\langle \eta(y), v_a\rangle : y\in\mathbb R\}$ where $K_{v_a} (u) =\langle II(u, u), v_a\rangle= II_{v_a}(u, u)$ is the {\it axial normal curvature function}.

For $k=2$, $\Delta_p$ is again a planar parabola that may degenerate into a line, a half-line or a point, however, this parabola now lies on a plane in $\mathbb R^3$. In \cite{Benedini/Sinha/Ruas}, the umbilic curvature was defined even when $\Delta_p$ is a non-degenerate parabola as the height of this plane.

 \subsection{Singular $3$-manifold in $\mathbb R^{k+3}$}
Let $M^3_{\sing}\subset \mathbb R^{k+3}$ be a 3-manifold with a singularity of corank 1 at $p\in M^3_{\sing}$.
Let $\mathcal B=\{\partial_x,\partial_y,\partial_z\}$ be a basis for $T_{q}\tilde{M}$ and $u=\alpha\partial_{x}+\beta\partial_{y}+\gamma\partial_{z}\in T_{q}\tilde{M}$. The coefficients and images of the first and second fundamental forms are defined analogously to the case of surfaces. 
In particular, for each normal vector $\nu\in N_pM^3_{\sing}$, the coefficients of $II_{\nu}$ in terms of local coordinates $(x,y,z)$ are:
$$
\begin{array}{c}
     l_{\nu}(q)=\langle \pi_2(f_{xx}),\nu\rangle(\phi(q)),\, m_{\nu}(q)=\langle \pi_2(f_{xy}),\nu\rangle(\phi(q)),\, 
     n_{\nu}(q)=\langle \pi_2(f_{yy}),\nu\rangle(\phi(q)),\\
     p_{\nu}(q)=\langle \pi_2(f_{zz}),\nu\rangle(\phi(q)),\, q_{\nu}(q)=\langle \pi_2(f_{xz}),\nu\rangle(\phi(q)),\, 
     r_{\nu}(q)=\langle \pi_2(f_{yz}),\nu\rangle(\phi(q)),\\
\end{array}
$$
Fixing an orthonormal frame $\{\nu_1,\ldots,\nu_{k+1}\}$ of $N_{p}M^{3}_{\sing}$, the second fundamental form can be written as
$$
\begin{array}{c}\label{eq.2ff-3man}
II(u,u) =\displaystyle\sum_{i=1}^{k+1}II_{\nu_{i}}(u,u)\nu_{i}.
\end{array}
$$

Taking a coordinate system $\phi$ and making rotations in the target, it is possible to obtain a local parametrisation in the Monge form
as in (\ref{Monge-final}).
Hence, the subset of unit tangent vectors $C_q$ is the cylinder given by $\{(\alpha,\beta,\gamma)\in T_q\tilde{M} \;:\; \alpha^2+\beta^2=1\}$ parallel to the $z$-axis. Taking an orthonormal frame $\{\nu_{1},\ldots,\nu_{k+1}\}$ of $N_{p}M^3_{\sing}$, the curvature locus $\Delta_{p}$ can be parametrised by $(\alpha,\beta,\gamma)\mapsto$
$$
\sum_{i=1}^{k+1}
(\alpha^{2}l_{\nu_i}(q)+2\alpha\beta m_{\nu_i}(q)+\beta^{2}n_{\nu_i}(q)+\gamma^{2}p_{\nu_i}(q)+2\alpha\gamma q_{\nu_i}(q)+2\beta\gamma r_{\nu_i}(q)
)\nu_{i}
$$
with $\alpha^2+\beta^2=1$.

\vspace{0.7cm}

Throughout the paper, $\mathscr A=Diff(\mathbb R^n,0)\times Diff(\mathbb R^{n+k},0)$, i.e. changes of coordinates in source and target, and $\mathscr A^2$ represents the 2-jets of elements in $\mathscr A$. All our results are local in the sense that we are considering germs of manifolds at a certain point.

\section{Definition of axial curvatures for $M^{n}_{\sing}$ in $\mathbb R^{n+k}$}\label{defs}

For the case of corank 1 surfaces $M^{2}_{\sing}\subset\mathbb R^3$ the parametrisation of $f$ can be given at the origin in Monge form
\begin{equation}\label{Monge2}
j^2f(0)=(x, \frac{1}{2}(a^1_{20}x^2 + 2a^1_{11}xy + a^1_{02}y^2), \frac{1}{2}(a^2_{20}x^2 + 2a^2_{11}xy + a^2_{02}y^2)).
\end{equation}
In \cite{OsetSaji}, when $\Delta_p$ is a non-degenerate parabola or a half-line the axial vector is defined using the direction perpendicular to the directrix of the parabola or the direction of the half-line and is given by 
$v_a=\frac{1}{\sqrt{(a^1_{02})^2+(a^2_{02})^2}} (a^1_{02}, a^2_{02})$.
The axial curvature is given by 
$$
\kappa_a(p)=\frac{1}{||(a^1_{02},a^2_{02})||}\left( (a^1_{02}a^1_{20}+a^2_{02}a^2_{20})-\frac{(a^1_{02}a^1_{11}+a^2_{02}a^2_{11})^2}{||(a^1_{02},a^2_{02})||^2}\right).
$$
For the case when the curvature parabola is a line or half-line, in \cite{MartinsBallesteros} the authors define the image by $\eta$ of the direction $y_{\infty}\in T_q\tilde M$ as the direction in which the curvature locus is not bounded. This was generalized for all $\mathscr A^2$-orbits in \cite{OsetSaji}. The following result which relates this direction with the axial vector, despite natural and not surprising, had been unnoticed so far. 

For $M^{n}_{\sing}\subset\mathbb R^{n+k}$ considering the pseudo-metric in $T_q\tilde{M}$, we call \emph{null tangent direction} the unitary tangent direction $u_{\infty}\in T_q\tilde{M}$ such that $I(u_{\infty},u_{\infty})=0$ (it corresponds to $y_{\infty}$ in the notation of \cite{MartinsBallesteros} for $M^{2}_{\sing}$ in $\mathbb R^{3}$).

\begin{prop} Let $M^{2}_{\sing}\subset\mathbb R^3$ be such that $\Delta_p$ is a non-degenerate parabola or a half-line. Let $u_{\infty}\in T_q\tilde{M}$ be the null tangent direction, then $\frac{II(u_{\infty},u_{\infty})}{|| II(u_{\infty},u_{\infty})||}=v_a$.
\end{prop}
\begin{proof}
Consider $M^{2}_{\sing}$ given by the image of $f$ in Monge form with $j^2f$ as in (\ref{Monge2}), then $E(q)=1$, $F(q)=G(q)=0$, and $u_{\infty}=(0,1)$ corresponds to the null tangent direction. Then $II(u_{\infty},u_{\infty})=(a_{02}^1,a_{02}^2)$.
\end{proof}


This result suggests how the axial vector can be defined for higher dimensions. However the idea of projecting the curvature locus to a certain direction in order to obtain meaningful curvatures should not be restricted to the axial vector. We shall define several axial vectors and, consequently, several axial curvatures.


\begin{definition}
Consider $M^{n}_{\sing}\subset\mathbb R^{n+k}$, then $\dim T_q\tilde{M}=n$ and $\dim N_pM^n_{\sing}=k+1$. Let $l=min\{n,k+1\}$. Define the \emph{axial space} $Ax_p\subset N_pM^n_{\sing}$ as $Aff_p$ if $dim Aff_p=l$ and as any $l$-vector space containing $Aff_p$ if $dim Aff_p< l$. In the second case $Ax_p$ contains $Aff_p$ and $p$ and in both cases $\dim Ax_p=l$ (as an affine or vector space respectively).
\end{definition}

Our goal is to define an adapted frame as in \cite{OsetSaji} for $Ax_p$. We start with a partial definition.

\begin{definition}\label{primaxialvect}
Let $u_{\infty}\in T_q\tilde{M}$ be the null tangent vector, and suppose $II(u_{\infty},u_{\infty})\neq 0$. Then the {\it primary axial vector} is
$$
v_a^1:=\displaystyle \frac{II(u_{\infty},u_{\infty})}{|| II(u_{\infty},u_{\infty})||}.
$$
In order to obtain an adapted frame for $Ax_p$ we add normal vectors in such a way that $\{v_a^1,\ldots,v_a^l\}$ is a positively oriented orthonormal frame.
\end{definition}

How to complete the basis and how to define an adapted frame when $II(u_{\infty},u_{\infty})=0$ will be defined separately for the cases $n=2,3$ in the next sections.

Although Definition \ref{primaxialvect} needs to be completed, we are already in position for our main definition.


\begin{definition}
Given an adapted frame $\{v_a^1,\ldots,v_a^l\}$ of $Ax_p\subset N_pM^n_{\sing}$, the \emph{$i$-ary normal curvature function} is given by 
$$
K_{v_a^i} (w) =\langle II(w, w), v_a^i\rangle= II_{v_a^i}(w, w),
$$
and the \emph{$i$-ary axial curvatures} are the numbers
$$
\kappa_{a_i}(p) =\text{critical values of   } K_{v_a^i}(w) \text{   where   } w\in C_q.
$$
\end{definition}
Taking a Monge form as in (\ref{Monge-final}) for a corank 1 singular manifold $M^{n}_{\sing}\subset\mathbb R^{n+k}$, the subset of unit tangent vectors is the cylinder given by $C_q=\{(x_1,\ldots,x_n)\in T_q\tilde{M}: x_1^2+\ldots+x_{n-1}^2=1\}\subset\mathbb R^n$ parallel to the $x_n$-axis. Under these conditions, we can prove the following.

\begin{prop}\label{numaxial}
There are at most $l(n-1)$ axial curvatures.
\end{prop}
\begin{proof}
Fix a certain axial vector $v_a^i$. To study the critical points of $K_{v_a^i} (w)$ when $w\in C_q$ we want all the $2\times 2$ minors of the following matrix to be 0:
\begin{equation*}
\left(
\begin{array}{cccc}
\frac{\partial K_{v_a^i}}{\partial x_1} & \ldots & \frac{\partial K_{v_a^i}}{\partial x_{n-1}} & \frac{\partial K_{v_a^i}}{\partial x_n} \\
2x_1 & \ldots & 2x_n & 0 \\
\end{array}
\right),
\end{equation*}
where in the second row we have the gradient of the equation for $C_q$. This is a $2\times n$ matrix with linear entries in $n$-variables. The solutions to the system given by the minors is a homogeneous algebraic variety which is generically a collection of lines through the origin. The intersection of this lines with the cylinder $C_q$ give the critical points of $K_{v_a^i} (w)$. Since the second fundamental form is quadratic homogeneous, antipodal points in the cylinder have the same image, so there are at most as many critical values as lines in the solution to the system.

On the other hand in Lemma 5.5 of \cite{bivianuno}, there is a formula for the multiplicity of the ideal generated by the $2\times 2$ minors. Applying this formula to our situation we obtain that the multiplicity is $n$, i.e. generically there can be up to $n$ lines through the origin of multiplicity 1 as a solution to our system. However, notice that $x_1=\ldots=x_{n-1}=0$ is always a solution, but the line $(0,\ldots,0,x_n)$ does not intersect the cylinder. Therefore, there are at most $n-1$ critical points and so at most $n-1$ critical values.

Since we have $l$ axial vectors, the result follows.
\end{proof}

\begin{rem}\label{casosupen3}
\begin{enumerate}
    \item[i)] When there is more than 1 axial curvature for each $i$ we will denote them by $\kappa_{a_i}^j$, $1\leq j\leq n-1$. On the other hand, it is possible that no axial curvature exists for a certain $i$. Notice that in this case, saying that an axial curvature exists is equivalent to saying it is finite.

\item [ii)] In the particular case of $M^{2}_{\sing}\subset\mathbb R^{3}$, the primary axial curvature coincides with the axial curvature $\kappa_a$ defined in \cite{OsetSaji} and the absolute value of the secondary axial curvature coincides with the umbilic curvature $\kappa_u$ defined in \cite{MartinsBallesteros} (see Section \ref{prelim}).
\end{enumerate}
\end{rem}

For a corank 1 singular $n$-manifold parametrised in Monge form as above, the null tangent vector is given by $u_{\infty}=(0,\ldots,0,1)$. Consider the immersed $(n-1)$-manifold in $\mathbb R^{n+k}$ given by $f(x_1,\ldots,x_{n-1},0)$. Then $T_qM^{n-1}_{\reg}=u_{\infty}^{\perp}=T_q\tilde{M}\cap\{x_n=0\}$, and the pseudo-metric induces a metric here. Let $A_{\nu}:T_qM^{n-1}_{\reg}\to T_qM^{n-1}_{\reg}$ be the associated shape operator along the normal vector field $\nu$ such that $\langle A_{\nu}(w),w\rangle=\langle II(w,w), \nu\rangle$, where $w\in u_{\infty}^{\perp}$. There exists an orthonormal basis $\{e_1,\ldots,e_{n-1}\}$ of $T_qM^{n-1}_{\reg}$ of eigenvectors of $A_{\nu}$ and the corresponding eigenvalues $\kappa_1^{\nu},\ldots,\kappa_{n-1}^{\nu}$ are the $\nu$-principal curvatures.

We will show in Section \ref{M3} that, for $n=3$ at least, when $n-1$ $i$-ary axial curvatures are finite they coincide with the eigenvalues of $A_{v_a^i}$, i.e. the $v_a^i$-principal curvatures of regular $(n-1)$-dimensional manifold $M^{n-1}_{\reg}$.

\section{Axial curvatures for $M^{2}_{\sing}$ in $\mathbb R^{2+k}$}\label{sups}

We start by defining the adapted frame for $Ax_p$. Here $\dim Ax_p=l=2$ so it is enough to define the primary axial vector in order to obtain an adapted frame. For surfaces, as in \cite{MartinsBallesteros} and \cite{Benedini/Sinha/Ruas}, the curvature locus can be a non-degenerate parabola, a half-line, a line or a point. Given the null tangent vector $u_{\infty}\in T_q\tilde{M}$, when $II(u_{\infty},u_{\infty})\neq 0$, $v_a^1$ is defined as in Definition \ref{primaxialvect} and we can complete the basis in a unique way such that $\{v_a^1,v_a^2\}$ is a positively oriented orthonormal frame of $Ax_p$. This includes the cases where $\Delta_p$ is a non-degenerate parabola or a half-line.

When $II(u_{\infty},u_{\infty})=0$, define $v_a^1$ as the direction of $\Delta_p$ when it is a line and if $\Delta_p$ is a point $y\neq p$, then $v_a^1$ is such that $v_a^2=y/|y|$. If $y=p$, then any orthonormal frame is an adapted frame. 

Given the nature of the curvature locus for singular surfaces (a non-degenerate parabola, a half-line, a line or a point), there will only be one axial curvature for each $i=1,2$. When $\Delta_p$ is a line $\kappa_{a_1}$ is not defined (there is no critical point of the primary axial normal curvature function). When $\Delta_p$ is a point $\kappa_{a_1}=0$. When $\Delta_p$ is a non-degenerate parabola, $\kappa_{a_2}$ is not defined (there is no critical point of the secondary axial normal curvature function). When $\Delta_p$ is a point $\kappa_{a_2}=||p||$. In general we can write 
$$
\kappa_{a_1}(p) = min\{K_{v_a^1}(w) : w\in C_q\},
$$
$$
\kappa_{a_2}(p) = K_{v_a^2}(w) \text{   for any   } w\in C_q.
$$  
See Figure \ref{fig1} for the case when $\Delta_p$ is a half-line.

\begin{figure}\label{fig1}
\centering
\includegraphics[scale=0.75]{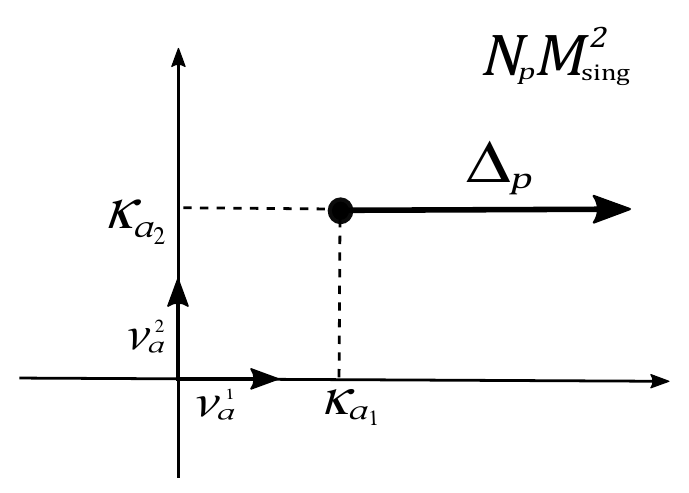}
\caption{Axial curvatures for $M_{\sing}^2\subset\mathbb R^3$ when $\Delta_p$ is a half-line.}
\end{figure}

\begin{prop}\label{axialcurvatureR4}
Let $p$ be the origin in $\mathbb R^{2+k}$ and $M^{2}_{\sing}$ be given by the image of $f$ in Monge form such that
\begin{equation}\label{MongeComplete}
j^2f(0)=(x, j^2f_1(0),\ldots,j^2f_{k+1}(0)),
\end{equation}
where $j^2f_\ell(0)=\frac{1}{2}(a^\ell_{20}x^2 + 2a^\ell_{11}xy + a^\ell_{02}y^2)$ for $\ell=1,\ldots,k+1$.
Denote ${\bf a_{ij}}=(a^1_{ij},\ldots, a^{k+1}_{ij})$, 
then $II(u_{\infty},u_{\infty})={\bf a_{02}}$, where $u_{\infty}$ is the null tangent direction. 

\begin{itemize}
\item[(a)] If ${\bf a_{02}}\neq 0$ (i.e. $\Delta_p$ is a non-degenerate parabola or a half-line) then
$$
\kappa_{a_1}(p)=\frac{1}{|| {\bf a_{02}}||}\left(\langle {\bf a_{02}},{\bf a_{20}}\rangle-\frac{\langle {\bf a_{02}},{\bf a_{11}}\rangle^2}{|| {\bf a_{02}}||^2}\right).
$$
Furthermore, if ${\bf a_{02}}\times {\bf a_{11}}=0$ (i.e. $\Delta_p$ is a half-line), then
$$\kappa_{a_2}(p)=\frac{||{\bf a_{02}}\times {\bf a_{20}}||}{||{\bf a_{02}}||}.$$

\item[(b)] If ${\bf a_{02}}= 0$ and ${\bf a_{11}}\neq 0$ (i.e. $\Delta_p$ is a line), then 
$$
\kappa_{a_2}(p)=
\frac{||{\bf a_{20}}\times {\bf a_{11}}||}{||{\bf a_{11}}||}.
$$
\end{itemize}
\end{prop}
\begin{proof}
The proof for $\kappa_{a_1}$ follows the same idea as the proof of Proposition 4.3 in \cite{OsetSaji}. 
When $f$ is given as above, the curvature locus is parameterised by
$$
\eta(y)=(a^1_{20} + 2a^1_{11}y + a^1_{02}y^2,\ldots,a^{k+1}_{20} + 2a^{k+1}_{11}y + a^{k+1}_{02}y^2)
$$
When ${\bf a_{02}}\neq 0$, the primary axial vector is given by $v^1_a=\displaystyle\frac{{\bf a_{02}}}{||{\bf a_{02}}||}$, so $K_{v_a^1}(y)=\langle\eta(y),v_a^1\rangle$.  A direct computation shows that $\displaystyle y_0=-\frac{\langle {\bf a_{02}},{\bf a_{11}}\rangle}{|| {\bf a_{02}}||}$ is the minimal critical point of $K_{v^1_a}$ and the primary axial curvature is given by $\kappa_{a_1}(p)=K_{v^1_a}(y_0)$. 

Moreover, as ${\bf a_{02}}\neq 0$ we can suppose without loss of generality that  $a^{k+1}_{02}\neq0$. Using smooth changes of coordinates in the source and isometries in the target we can reduce $j^2f$ to the form
\begin{eqnarray}\label{special}
\begin{array}{l}
j^2f_\ell(q)=\frac12(\bar{a}^\ell_{20}x^2 + 2\bar{a}^\ell_{11}xy) \mbox{ for $\ell=1,\ldots,k$}\\
j^2f_{k+1}(q)=\frac12(\bar{a}^{k+1}_{20}x^2 + 2\bar{a}^{k+1}_{11}xy+\bar{a}^{k+1}_{02}y^2).
\end{array}
\end{eqnarray}


If ${\bf a_{02}}\times {\bf a_{11}}=0$ , we can reduce (\ref{special}), using smooth changes of coordinates in the source and isometries in the target, to the form
$j^2f_\ell(q)=0$  for $\ell=1,\ldots,k-1$, $j^2f_{k}(q)=\frac12\tilde{{a}}^{k}_{20}x^2$,  and $j^2f_{k+1}(q)=\frac12(\tilde{{a}}^{k+1}_{20}x^2 + 2\tilde{{a}}^{k+1}_{11}xy + \tilde{{a}}^{k+1}_{02}y^2)$, 
where $$\tilde{{a}}^{k}_{20}=-\frac{||{\bf a_{20}}\times {\bf a_{02}}||}{||{\bf a_{02}}||}.$$ We have $v^1_a=(0,\ldots,0,1)$, $v_a^2=(0,\ldots,0,-1,0)$ and 
$$
\eta(y)=(0,\ldots,0,\tilde{{a}}^{k}_{20},\tilde{{a}}^{k+1}_{20} + 2\tilde{{a}}^{k+1}_{11}y + \tilde{{a}}^{k+1}_{02}y^2),
$$ 
Therefore $\kappa_{a_2}=-\tilde{{a}}^{k}_{20}$.

(b) When ${\bf a_{02}}= 0$ and ${\bf a_{11}}\neq 0$, then $II(u,u)=0$. We can suppose without loss of generality that  $a^{1}_{11}\neq0$ then using smooth changes of coordinates in the source and isometries in the target, we can reduce $j^2f$ to the form
$j^2f_{1}(q)=\frac12(\bar{a}^{1}_{20}x^2+2\bar{a}^{1}_{11}xy)$, $j^2f_{2}(q)=\frac12\bar{a}^{2}_{20}x^2$, and  $j^2f_\ell(q)=0$ for $\ell=3,\ldots,k+1$, 
where $$\bar{a}^{2}_{20}=\frac{||{\bf a_{20}}\times {\bf a_{11}}||}{||{\bf a_{11}}||}.$$ Here the primary axial vector is $v^1_a=(1,0,\ldots,0)$ we obtain  $v_a^2=(0,1,0\ldots,0)$ and $k_{a_2}(p)=K_{v_a^2}(w)=\bar{a}^{2}_{20}$.

\end{proof}

\begin{rem}
\begin{itemize}
\item[i)] In the previous proof the isometries may change the orientation of the basis of $N_pM^2_{\sing}$. On the other hand, the adapted frame is constructed using the locus. So the sign of $\kappa_{a_2}$ may change if the new orientation of the basis of $N_pM^2_{\sing}$ does not coincide with the positive orientation of the adapted frame.

\item[ii)] The previous formulas confirm what was announced in Remark \ref{casosupen3}, i.e. for the particular case of $M^2_{\reg}\subset \mathbb R^3$ $\kappa_{a_1}$ and $\kappa_{a_2}$ recover the axial and umbilic curvatures in \cite{OsetSaji} and \cite{MartinsBallesteros}. In fact, the formulas for $\kappa_{a_2}$ give explicit formulas for $\kappa_u$ which were not given in \cite{MartinsBallesteros}.

\end{itemize}
\end{rem}

We now give some geometrical interpretations for the axial curvatures.

\begin{prop}
Let $f$ be the parametrisation of a corank 1 singular surface in $\mathbb R^4$ such that $j^2f\sim_{\mathscr A^2}(x,y^2,0,0)$. Then the curvature $\kappa$ of the regular curve $\gamma(t)=f(t,0)$ at the origin satisfies  $\kappa^2=(\kappa_{a_1})^2+(\kappa_{a_2})^2$.
\end{prop}
\begin{proof}
As $j^2f\sim_{\mathcal A}(x,y^2,0,0)$ we can take $f$ parametrised by 
$$
(x,y)\mapsto\left(x,\frac{a^1_{20}}{2}x^2 +y^2 + p(x,y),\frac{a_{20}^2}{2}x^2+q(x,y),r(x,y)\right)
$$
where $p,q,r\in \mathcal M_2^3$ (see Lemma 3.7 in \cite{Benedini/Sinha/Ruas}). Here the curvature parabola is parameterised by $\eta(y)=(a_{20}^1 +2y^2,a_{20}^2,0)$, then $\kappa_{a_2}=a_{20}^2$ and by Proposition \ref{axialcurvatureR4} we get $\kappa_{a_1} =a_{20}^1$. Furthermore, when $y=0$ the curvature of the regular curve $\gamma(t)=f(t,0)$  satisfies the desired result.  
\end{proof}

\begin{rem}
This formula generalizes the formula given in \cite[p. 455]{MartinsSaji} for the case of frontals in $\mathbb R^3$ where it was proved that $\kappa^2=\kappa_{s}^2+\kappa_{\nu}^2$, where $\kappa$ is the curvature of the cuspidal edge curve, $\kappa_{s}$ is the singular curvature and $\kappa_{\nu}$ is the limiting normal curvature.
\end{rem}

\begin{ex}\label{formulas}
Consider the singular surface parameterised by 
$$
f (x, y) = \left(x, \frac{a_{20}^1}{2}x^2 + y^2, \frac{a^2_{20}}{2}x^2,\frac{a^3_{20}}{2}x^2 +y^3\right)
$$ 
which can be seen as a surface in $\mathbb R^4$ with a cuspidal edge. The curvature parabola is a half-line parameterised by $\eta(y)=(a_{20}^1 +2y^2,a_{20}^2,a_{20}^3)$.
By Proposition \ref{axialcurvatureR4}, $\kappa_{a_1}(p)=a_{20}^1$ and $\kappa_{a_2}(p)=\sqrt{(a_{20}^2)^2+(a_{20}^3)^2}$. The regular curve $\gamma(t)=f(t,0)$ has curvature $\kappa=||{\bf a_{20}}||$, and hence $\kappa^2=(\kappa_{a_1})^2+(\kappa_{a_2})^2$.
\end{ex}

Given $M^2_{\sing}\subset\mathbb{R}^4$ a singular corank one surface at $p$, one can associate a regular surface $M^2_{\reg}\subset\mathbb{R}^4$. In \cite[p. 782]{Benedini/Sinha/Ruas} the authors provide such construction and prove that their second order geometries are strongly related to each other, since they have the same second fundamental form (see Theorem 4.14 in \cite{Benedini/Sinha/Ruas}). The singular surface can be obtained as a projection of a regular surface $N^2_{\reg}\subset\mathbb{R}^5$ in a tangent direction, via the map
$\xi:N^2_{\reg}\subset\mathbb{R}^{5}\rightarrow M^{2}_{\sing}\subset\mathbb{R}^4$. The regular surface $N^2_{\reg}\subset\mathbb{R}^{5}$ can be taken, locally, as the image of an immersion $i:\tilde{M}\rightarrow N^{2}_{\reg}\subset\mathbb{R}^{5}$, where $\tilde{M}$ is the regular surface from the construction in Section \ref{prelim} ($k=n=2$).

Hence, $M^{2}_{\reg}\subset\mathbb{R}^{4}$ is the regular surface locally obtained by projecting $N^{2}_{\reg}\subset\mathbb{R}^{5}$ via the map $\texttt{p}$ into the four space given by $T_{\xi^{-1}(p)}N^2_{\reg}\oplus \xi^{-1}(E_{p})$, where $E_p$ is a plane through $p$ parallel to $Aff_p$ (see the following diagram).

$$
\xymatrix{
 &  & N^2_{\reg}\subset\mathbb{R}^{5}\ar[rd]^-{\texttt{p}}\ar[d]^-{\xi} & \\
 \mathbb{R}^{2}\ar@/_0.7cm/[rr]^-{f}& \tilde{M}\ar[l]_-{\phi}\ar[r]^-{g}\ar[ru]^-{i}& M^2_{\sing}\subset\mathbb{R}^{4}& M^2_{\reg}\subset\mathbb{R}^{4}
}
$$

\begin{prop}
Let $M^2_{\sing}\subset\mathbb{R}^4$ a corank 1 surface at $p$ and $M^2_{\reg}\subset\mathbb{R}^4$ its associated regular surface. Suppose $M^2_{\sing}$ locally parametrised as in (\ref{MongeComplete}), with ${\bf a_{02}}\neq 0$. The Gaussian curvature of $\pi(M^2_{\reg})\subset\mathbb{R}^3$, the regular surface obtained projecting $M^2_{\reg}$ along $(v_a^1)^{\perp}$, at the respective point, is given by $K=\kappa_{a_1}(p)$.
\end{prop}
\begin{proof}
Since ${\bf a_{02}}\neq 0$, we can change the local parametrisation of $M^2_{\sing}\subset\mathbb{R}^4$ to the form (\ref{special}), where $k=2$.
Furthermore, using a rotation, it is possible to eliminate $\bar{a}^\ell_{20}$, for $\ell=1$ or 2. Without loss of generality, suppose $\bar{a}^1_{20}=0$. Finally, we take the change of coordinates in the source given by $(x,y)\mapsto(x,\frac{1}{||{\bf a_{02}}||}y)$. Therefore, $E_p$ is the plane generated by the last two coordinate axes, $M^2_{\reg}\subset\mathbb{R}^4$ is locally given by
$(x,y,j^2f_1(q),j^2f_2(q))$ and $v_a^1=(0,0,0,1)$. The Gaussian curvature of $\pi(M^2_{\reg})\subset\mathbb{R}^3$, at the respective point, is given by
$K=\bar{a}^1_{20}\bar{a}^1_{02}-(\bar{a}^1_{11})^2=\kappa_{a_1}(p)$.
\end{proof}

Proposition 5.2 in \cite{OsetSaji} can be easily extended for $\mathbb R^n$. 
\begin{prop} If $f$ satisfies that $\Delta_p$ is a non-degenerate parabola or a half-line,
the $\mathscr A$-singularities of $h_{v_a^1}$, the height function in the direction $v_a^1$, are
\begin{enumerate}
    \item $A^+_1$ if and only if $\kappa_{a_1}(p) > 0$,
    \item $A^-_1$ if and only if $\kappa_{a_1}(p) < 0$,
    \item $A_{\geq2}$ if and only if $\kappa_{a_1}(p) =0$.
\end{enumerate}
\end{prop}
\begin{proof}
The proof is analogous to the proof in \cite{OsetSaji}. Following Lemma 3.6 and Lemma 3.7 in \cite{OsetSaji}, for any smooth map $f:\mathbb R^2\to \mathbb R^{k+2}$ with $q\in\mathbb R^2$ a corank 1 singular point of $f$ there exists a coordinate system $(x,y)$ which satisfies $f_x(q)\neq 0$, $f_y(q) = 0$, $|f_x(q)| =
|f_{xx}(q)| = 1$ and $\langle f_x(q), f_{yy}(q)\rangle= 0$. In such a coordinate system $$
\kappa_{a_1}(p) =( \langle f_{xx}, f_{yy}\rangle - \langle f_{xy}, f_{yy} \rangle^2 )(q). 
$$
The right hand side of the formula does not depend on the coordinate system as long as it satisfies the above conditions, so in particular one can chose $f$ such that
 $j^2f(0)=(x, j^2f_1(0),\ldots,j^2f_{k+1}(0)),$
where 
$$j^2f_\ell(0)=\frac{1}{2}\left(a^\ell_{20}x^2 + \frac{2a^\ell_{11}}{\sqrt{|| {\bf a_{02}}||}}xy + \frac{a^\ell_{02}}{|| {\bf a_{02}}||}y^2\right)
$$ 
for $\ell=1,\ldots,k+1$ ($\bf a_{02}\neq 0$ because $\Delta_p$ is a non-degenerate parabola or a half-line). Now consider contact with the plane orthogonal to $v_a^1$. This contact is measured by the $\mathscr A$-singularity of the height function $h_{v_a^1}(f(q))=\langle f(q),v_a^1\rangle$. Direct computation shows 
$$
\kappa_{a_1}(p) =\Hess(h_{v_a^1}(f(q)))
$$
and the result follows.
\end{proof}

We will give a geometrical interpretation for $\kappa_{a_2}$ in Section \ref{umbs}.


\section{Axial curvatures for $M^{3}_{\sing}$ in $\mathbb R^{3+k}$}\label{M3}

\subsection{Curvature loci and adapted frames}

As in the previous section, we start by defining the adapted frame for $Ax_p$. When $k=1$, then $l=2$, if $k>1$, then $l=3$ and so we must distinguish these two possibilities. 

Given a smooth map $f:\mathbb R^3\to \mathbb R^{3+k}$ with $q\in\mathbb R^3$ a corank 1 singular point of $f$, there exists a coordinate system in Monge form such that  
\begin{eqnarray}\label{Monge3man}
f(x,y,z)=(x,y,f_1(x,y,z),\ldots,f_{k+1}(x,y,z))
\end{eqnarray}
where $j^2f_\ell(q)=\frac{1}{2}(a^\ell_{200}x^2 + 2a^\ell_{110}xy + a^\ell_{020}y^2 + 2a^\ell_{101}xz + 2a^\ell_{011}yz+a^\ell_{002}z^2)$ for $\ell=1,\ldots,k+1$.
Consider the notation ${\bf a_{pqr}}=(a^1_{pqr},\ldots, a^{k+1}_{pqr})$ with $p,q,r=0,1,2$ and the matrix
\begin{equation}\label{matrixA}
  A=\left[\begin{array}{ccc}{\bf a_{101}}& {\bf a_{011}} & {\bf a_{002}}\end{array}\right]. 
\end{equation}

We start with the case of corank 1 3-manifolds in $\mathbb R^4$. In order to define the adapted frame we need to understand the types of curvature locus that can appear. For this we classify first the 2-jet orbits under $\mathscr A$-equivalence. We denote by $J^2(3,3+k)$ the subspace of 2-jets $j^2f(0)$ of map germs $f:(\mathbb R^3,0)\to(\mathbb R^{3+k},0)$ and by $\Sigma^1 J^2(3,3+k)$ the subset of 2-jets of corank 1.

\begin{prop}\label{A2 orbits}
There are five $\mathscr{A}^2$-orbits in $\Sigma^1J^2(3,4)$:
$$(x,y,xz,z^2),\ (x,y,xz,yz),\ (x,y,z^2,0),\ (x,y,xz,0)\ \mbox{and}\ (x,y,0,0).$$
\end{prop}
\begin{proof}
Let $f:(\mathbb{R}^3,0)\rightarrow(\mathbb{R}^4,0)$ be
a corank one map germ given by Monge form as in (\ref{Monge3man}) with $k=1$.
Take coordinates $(X,Y,Z,W)$ in the target. The change in the target given by $X'=X$, $Y'=Y$, $Z'=Z-a_{200}^1X^2-a_{020}^1Y^2-a_{110}^1XY$ and $W'=W-a_{200}^2X^2-a_{020}^2Y^2-a_{110}^2XY$ removes the terms with $x^2$, $y^2$ and $xy$ of the last two coordinates. Hence,
$$
j^2f(0)\sim_{\mathscr{A}^2}(x,y,a^1_{101}xz + a^1_{011}yz + a^1_{002}z^2,a^2_{101}xz + a^2_{011}yz + a^2_{002}z^2).
$$
Here the matrix $A$ is given by 
\begin{equation}
  A=\left(
\begin{array}{ccc}
a^1_{101} & a^1_{011} & a^1_{002} \\
a^2_{101} & a^2_{011} & a^2_{002} \\
\end{array}
\right).
\label{matriz A}
\end{equation}
and consider the $2\times 2$ minors $\alpha=a_{101}^1a_{011}^2-a^1_{011}a^2_{101}$, $\beta=a^1_{011}a^2_{002}-a^1_{002}a^2_{011}$ and $\gamma=a^1_{101}a^2_{002}-a^1_{002}a^2_{101}$.
We will consider the case $\rank(A)=2$. The remaining cases are analogous. If $||{\bf a_{002}}||\neq0$, suppose $a^1_{002}\neq0$.
After the change in the target $W''=-W'+\frac{a^2_{002}}{a^1_{002}}Z'$ (keeping the others coordinates unchanged), in the source $\tilde{z}=z-\frac{a^1_{101}}{2a^1_{002}}x-\frac{a^1_{011}}{2a^1_{002}}y$ and eliminating the terms with $x^2$, $y^2$ and $xy$ of the last two coordinates, we obtain
$$
j^2f(0)\sim_{\mathscr{A}^2}(x,y,a^1_{002}z^2,\frac{\gamma}{a^1_{002}}xz+\frac{\beta}{a^1_{002}}yz)\sim_{\mathscr{A}^2}(x,y,z^2,\gamma xz+\beta yz).
$$
Notice that $(\beta,\gamma)\neq(0,0)$, otherwise since $(a^1_{002},a^2_{002})\neq(0,0)$, we would have also $\alpha=0$. The change in the source given by $x'=x-\frac{\beta}{\gamma}y$ (assuming $\gamma\neq0$) followed by the change in the target $\tilde{X}=X'+\frac{\beta}{\gamma}Y'$, $\tilde{W}=\frac{1}{\gamma}W''$ provides us that
$j^2f(0)\sim_{\mathscr{A}^2}(x,y,xz,z^2)$. However, if $(a^1_{002},a^2_{002})=(0,0)$, $\alpha\neq0$ and $(a^{1}_{101},a^2_{101})\neq(0,0)$. Considering $a^{1}_{101}\neq0$, the change in the target $W'=W-\frac{a^{2}_{101}}{a^{1}_{101}}Z$ provides that
$$j^2f(0)\sim_{\mathscr{A}^2}(x,y,a^1_{101}xz+a^1_{011}yz,\frac{\alpha}{a^1_{101}}yz)\sim_{\mathscr{A}^2}(x,y,xz+\frac{a^1_{011}}{a^1_{101}}yz,yz).$$
Finally, the change in the target $Z'=Z-\frac{a^1_{011}}{a^1_{101}}W$ provides $j^{2}f(0)\sim_{\mathscr{A}^2}(x,y,xz,yz)$.
\end{proof}

\begin{rem}
Considering $f:(\mathbb{R}^3,0)\rightarrow(\mathbb{R}^4,0)$
a corank one map germ given in Monge form as in (\ref{Monge3man}) and the the matrix $A$ given as in (\ref{matriz A}), 
Table \ref{A2 conditions} presents conditions on the coefficients
to identify when the $2$-jet is equivalent to one of the five normal forms of Proposition \ref{A2 orbits}.

\begin{table}[h]
\caption{Conditions over the coefficients of the $2$-jet for the $\mathscr{A}^{2}$-classification of corank 1 map germ $(\mathbb{R}^3,0)\rightarrow(\mathbb{R}^4,0)$.}
\centering
{\begin{tabular}{ccc}
\hline
$\mathscr{A}^{2}$-normal form & Conditions\\
\hline
$(x,y,xz,z^2)$ &  $\rank(A)=2\,\,\, \mbox{and}\,\,\, ||{\bf a_{002}}||>0$ \\ \cr
$(x,y,xz,yz)$ & $\rank(A)=2\,\,\, \mbox{and}\,\,\, ||{\bf a_{002}}||=0$ \\ \cr
$(x,y,z^2,0)$ & $\rank(A)=1\,\,\, \mbox{and}\,\,\, ||{\bf a_{002}}||>0$\\ \cr
$(x,y,xz,0)$ & $\rank(A)=1\,\,\, \mbox{and}\,\,\, ||{\bf a_{002}}||=0$ \\ \cr
$(x,y,0,0)$ & $\rank(A)=0$. \cr
\hline
\end{tabular}
}
\label{A2 conditions}
\end{table}
\end{rem}

\begin{lem}\label{change}
Let $M^{3}_{\sing}$ be a corank 1 surface in $\mathbb R^{3+k}$ given by the image of Monge form $f$ as (\ref{Monge3man}). If $|| {\bf a_{002}}||\neq0$ then using smooth changes of coordinates in the source and isometries in the target, we can reduce $j^2f$ to the form
\begin{equation}\label{Monge4}
j^2f_\ell(q)=\frac12(\bar{a}^\ell_{200}x^2 + 2\bar{a}^\ell_{110}xy + \bar{a}^\ell_{020}y^2 + 2\bar{a}^\ell_{101}xz + 2\bar{a}^\ell_{011}yz) \mbox{ for $\ell=1,\ldots,k$}
\end{equation}

and
\begin{equation}\label{Monge5}
j^2f_{k+1}(q)=\frac12(\bar{a}^{k+1}_{200}x^2 + 2\bar{a}^{k+1}_{110}xy + \bar{a}^{k+1}_{020}y^2 + z^2), 
\end{equation}
where
$$
\begin{array}{l}
\displaystyle\bar{a}^{k+1}_{200}=\frac{1}{|| {\bf a_{002}}||}\left(\langle {\bf a_{002}},{\bf a_{200}}\rangle-\frac{\langle {\bf a_{002}},{\bf a_{101}}\rangle^2}{|| {\bf a_{002}}||^2}\right),\\
\displaystyle\bar{a}^{k+1}_{020}=\frac{1}{|| {\bf a_{002}}||}\left(\langle {\bf a_{002}},{\bf a_{020}}\rangle-\frac{\langle {\bf a_{002}},{\bf a_{011}}\rangle^2}{|| {\bf a_{002}}||^2}\right),\\
\displaystyle\bar{a}^{k+1}_{110}=\frac{1}{|| {\bf a_{002}}||}\left(\langle {\bf a_{002}},{\bf a_{110}}\rangle-\frac{\langle {\bf a_{002}},{\bf a_{101}}\rangle\langle {\bf a_{002}},{\bf a_{011}}\rangle}{|| {\bf a_{002}}||^2}\right).
\end{array}
$$
Moreover, when $rank(A)=1$, where $A$ is given as in (\ref{matrixA}), the coefficients $\bar{a}^\ell_{101}$ and $\bar{a}^\ell_{011}$ are zero for $\ell=1,\ldots,k$.
\end{lem}
\begin{proof} 
Consider $|| {\bf a_{002}}||\neq0$. Suppose, without loss of generality, that $a^{k+1}_{002}\neq0$. Taking  the  rotation in $\mathbb R^{3+k}$ of angle  $\gamma=\displaystyle\arctan\left(\frac{a^{1}_{002}}{a^{k+1}_{002}}\right)$ we can eliminate the coefficient of $z^2$ of $f_{1}$. In this case, we denote by $\tilde{f}$ the new normal form and by $\tilde{a}^\ell_{ijk}$ the coefficients of its 2-jet. 
After successive rotations in $\mathbb R^{k+3}$ with angle $\gamma=\displaystyle\arctan\left(\frac{\tilde{a}^{m}_{002}}{\tilde{a}^{k+1}_{002}}\right)$, $m=2,\ldots,k$, we eliminate all the coefficients of $z^2$ of the normal form, except in the last coordinate. So the 2-jet is
$$
\frac12(\tilde{a}^\ell_{200}x^2 + 2\tilde{a}^\ell_{110}xy + \tilde{a}^\ell_{020}y^2 + 2\tilde{a}^\ell_{101}xz + 2\tilde{a}^\ell_{011}yz)
$$
in coordinates $\ell=1,\ldots,k$ and
$$
\frac12(\tilde{a}^{k+1}_{200}x^2 + 2\tilde{a}^{k+1}_{110}xy + \tilde{a}^{k+1}_{020}y^2 + 2\tilde{a}^{k+1}_{101}xz + 2\tilde{a}^{k+1}_{011}yz+\tilde{a}^{k+1}_{002}z^2)
$$
in the last coordinate. Considering the changes of coordinates in the source $\displaystyle z= z'- \frac{\tilde{a}^{k+1}_{101}}{2\tilde{a}^{k+1}_{002}}x- \frac{\tilde{a}^{k+1}_{011}}{2\tilde{a}^{k+1}_{002}}y$ and then $z''=\frac{1}{\sqrt{|| {\bf a_{002}}||}}z'$ we obtain the desired normal form. 

Moreover, when $rank(A)=1$ we have that $\displaystyle a^\ell_{101}a^{k+1}_{002}-a^\ell_{002}a^{k+1}_{101}=0$ and $\displaystyle a^\ell_{011}a^{k+1}_{002}-a^\ell_{002}a^{k+1}_{011}=0$ for $\ell=1,\ldots,k$. Since the coefficients $\bar{a}^\ell_{101}$ and $\bar{a}^\ell_{011}$ contains these components as a factor in their expressions, they are 0 for $\ell=1,\ldots,k$.
\end{proof}

\begin{rem}
Taking normal sections of $M^{3}_{\sing}$ with the normal form from Lemma \ref{change} we obtain a singular surface $M^{2}_{\sing}\subset\mathbb R^{k+3}$ with primary axial vector given by $v^1_{a}=(0,\ldots,0,1)$. Notice that the coefficients $\bar{a}^{k+1}_{200}$ and $\bar{a}^{k+1}_{020}$ coincide with the primary axial curvatures of the normal sections obtained by $\{y=0\}$ and $\{x=0\}$ respectively (see Proposition \ref{axialcurvatureR4}).
\end{rem}

An analysis of the conditions in Table 1 can shred some light on the type of loci we can have in each orbit by following the ideas in the proof of Theorem 3.9 in \cite{BenediniRuasSacramento}, however, there is a more geometrical way of doing this. Consider a tangent direction $u\in T_pM^3_{\sing}$, we call the singular surface $M^3_{\sing}\cap\{u=0\}$ the normal section of $M^3_{\sing}$ in the direction $u$. Following the proof of Theorem 3.3 in \cite{BenediniOset2}, we have that the curvature locus of $M^3_{\sing}$ is generated by the union of the curvature loci of the normal sections. All the normal sections are corank 1 singular surfaces in $\mathbb R^3$ and the type of locus which can appear have been studied in \cite{MartinsBallesteros}. We can use this information to get the following.

\begin{prop}\label{types34}
Let $M^{3}_{\sing}\subset\mathbb R^4$ be parametrised by $f$, then 
\begin{enumerate}
    \item[i)]  $j^2f(0)\sim_{\mathscr{A}^2} (x,y,xz,z^2)$ if and only if $\Delta_p$ is a planar region,
    \item[ii)]  $j^2f(0)\sim_{\mathscr{A}^2} (x,y,xz,yz)$ if and only if $\Delta_p$ is a plane,
    \item[iii)]  $j^2f(0)\sim_{\mathscr{A}^2} (x,y,z^2,0)$ if and only if $\Delta_p$ is a half-strip (which may degenerate to a half-line),
    \item[iv)]  $j^2f(0)\sim_{\mathscr{A}^2} (x,y,xz,0)$ if and only if $\Delta_p$ is a strip (which may degenerate to a line),
    \item[v)]  $j^2f(0)\sim_{\mathscr{A}^2} (x,y,0,0)$ if and only if $\Delta_p$ is the curvature locus of a regular surface in $\mathbb R^4$ (ellipse, segment or point).
\end{enumerate}

\end{prop}
\begin{proof}
When $j^2f(0)\sim_{\mathscr{A}^2} (x,y,xz,z^2)$, all normal sections of type $y=\lambda x$ give corank one surfaces whose 2-jet is $\mathscr{A}^2$-equivalent to $(x,xz,z^2)$. The curvature locus of these sections is a non-degenerate parabola with axial vector $(0,1)$ in the normal plane. The normal section $x=0$ gives a 2-jet $\mathscr{A}^2$-equivalent to $(y,0,z^2)$, whose curvature locus is a half line in the direction $(0,1)$. The union of all these curvature loci gives a ``parabolic" planar region. This region is not the whole plane because the 2-jet of the last component of the parametrisation of the curvature locus by Lemma \ref{change} can be taken to $a^2_{200}x^2 + 2a^2_{110}xy + a^2_{020}y^2+z^2$ where $x^2+y^2=1$, which is a bounded function plus $z^2$, which is positive, so it is bounded on the bottom.

When $j^2f(0)\sim_{\mathscr{A}^2} (x,y,xz,yz)$, the normal section $x=0$ gives a 2-jet equivalent to $(y,0,yz)$, whose curvature locus is a line in the direction $(0,1)$ and the section $y=0$ gives a 2-jet equivalent to $(x,xz,0)$, whose curvature locus is a line in the direction $(1,0)$. The rest of normal sections give lines in any direction between $(0,1)$ and $(1,0)$, so the curvature locus of the 3-manifold is the whole plane.

When $j^2f(0)\sim_{\mathscr{A}^2} (x,y,z^2,0)$, all normal sections have a half-line in the direction $(1,0)$ as curvature loci. By Lemma \ref{change}, the curvature locus of the 3-manifold can be taken to $(a^1_{200}x^2 + 2a^1_{110}xy + a^1_{020}y^2+z^2,a^2_{200}x^2 + 2a^2_{110}xy + a^2_{020}y^2)$ where $x^2+y^2=1$. The curvature locus is bounded in the direction $(0,1)$ because the last component of the curvature locus is a bounded function. On the other hand the component in the direction $(1,0)$ is a bounded function plus $z^2$, which is positive, so it is bounded on the left. We therefore have a strip bounded on the left, which can degenerate to a half-line.

When $j^2f(0)\sim_{\mathscr{A}^2} (x,y,xz,0)$, all normal sections have lines in the direction $(1,0)$ as curvarutre loci, except for the section $x=0$, whose curvature locus is a point. The curvature locus of the 3-manifold is a strip unbounded in the direction $(1,0)$. By item e) in Proposition 4.13 in \cite{BenediniRuasSacramento} adapted for $k=1$ (the proof is the same), the last component of the curvature locus can be taken to $a^2_{200}x^2 + 2a^2_{110}xy + a^2_{020}y^2$ where $x^2+y^2=1$, so the curvature locus is bounded in the direction $(0,1)$.

When $j^2f(0)\sim_{\mathscr{A}^2} (x,y,0,0)$, by item f) in Proposition 4.13 in \cite{BenediniRuasSacramento} adapted for $k=1$ (the proof is the same), the curvature locus coincides with the curvature locus of a parametrisation of type $(x,y,\frac12(a^1_{200}x^2 + 2a^1_{110}xy + a^1_{020}y^2),\frac12(a^2_{200}x^2 + 2a^2_{110}xy + a^2_{020}y^2))$, and so the curvature locus can be any type of curvature locus of a regular surface in $\mathbb R^4$, i.e. a non-degenerate ellipse, a segment or a point.
\end{proof}

\begin{ex}\label{ex.3var4}
We shall present examples of curvature locus for each possibility in Proposition \ref{types34}. Let $M^{3}_{\sing}\subset\mathbb R^4$ be locally parametrised by $f:(\mathbb{R}^3,0)\rightarrow(\mathbb{R}^4,0)$.
\begin{enumerate}
    \item[i)] Taking $f(x,y,z)=(x,y,\frac{3}{2}x^2+xy+\frac{1}{2}y^2+\frac{1}{2}z^2,x^2+\frac{5}{2}y^2+\frac{1}{2}xz)$, $\Delta_p$ is a planar region (Figure \ref{orbitas} Planar region);
    \item[ii)] Taking $f(x,y,z)=(x,y,\frac{3}{2}x^2+\frac{1}{2}xy+\frac{1}{2}y^2+\frac{1}{2}yz,x^2+\frac{5}{2}y^2+\frac{1}{2}xz)$, $\Delta_p$ is a plane (Figure \ref{orbitas} Plane);
    \item[iii)] For $f(x,y,z)=(x,y,\frac{3}{2}x^2+xy+\frac{1}{2}y^2+\frac{1}{2}z^2,x^2+\frac{5}{2}y^2)$, $\Delta_p$ is a half-strip (Figure \ref{orbitas} Half-strip);
    \item[iv)] For $f(x,y,z)=(x,y,\frac{3}{2}x^2+\frac{1}{2}xy+\frac{1}{2}y^2,x^2+\frac{5}{2}y^2+\frac{1}{2}xz)$ $\Delta_p$ is a strip (Figure \ref{orbitas} Strip);
    \item[v)] Finally, taking $f(x,y,z)=(x,y,\frac{3}{2}x^2+xy+\frac{1}{2}y^2,x^2+\frac{5}{2}y^2)$, $\Delta_p$ is an ellipse (Figure \ref{orbitas} Ellipse).
\end{enumerate}
The curvature loci are not completely depicted in Figure \ref{orbitas}, the planar region and the half-strip should be extended infinitely on the right, the figure called plane should be extended infinitely in every direction and the strip should be extended infinitely up and down. 

\begin{figure}[!htb]
\includegraphics[height=7cm,width=12cm]{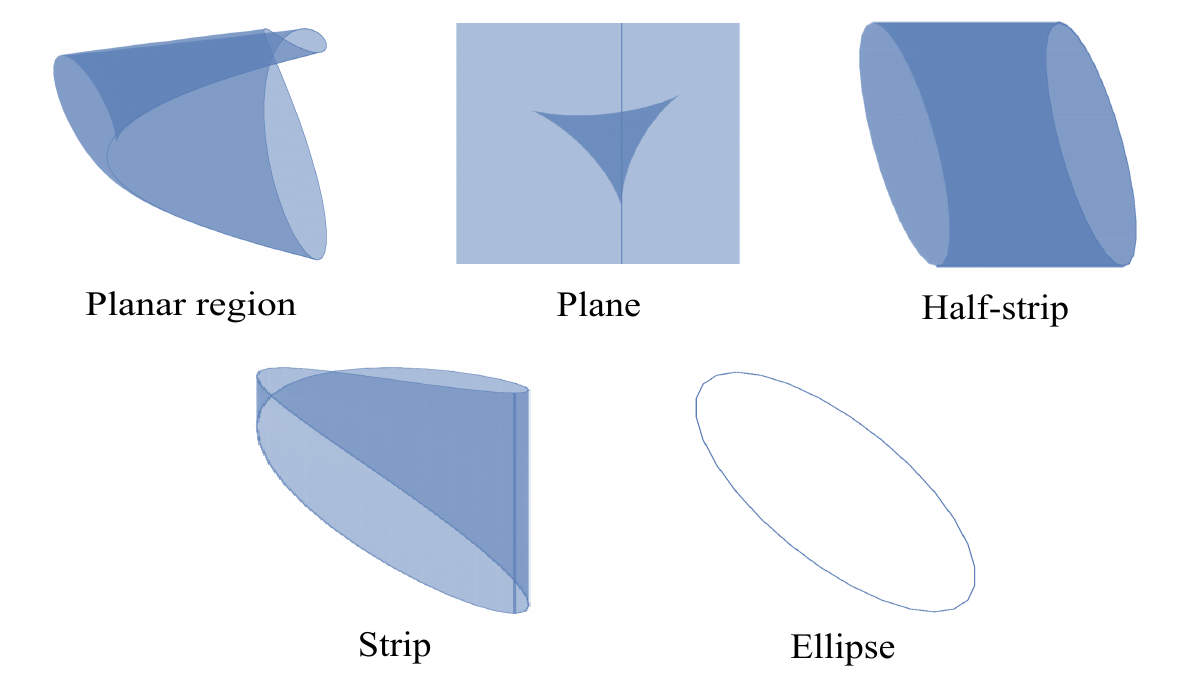}
\caption{Curvature loci for the different orbits for $M^{3}_{\sing}\subset\mathbb R^4$.}
\label{orbitas}
\end{figure}

\end{ex}

Now we can define our adapted frame for $Ax_p$. Here $l=2$, so it is enough to define the primary axial vector.

\begin{enumerate}
    \item [i)] When $II(u_{\infty},u_{\infty})\neq 0$, $v_a^1$ is defined as in Definition \ref{primaxialvect} and we can complete the basis in a unique way such that $\{v_a^1,v_a^2\}$ is a positively oriented orthonormal frame of $Ax_p$. This includes the first and third orbits in Proposition \ref{A2 orbits}.
    \item[ii)] If $II(u_{\infty},u_{\infty})=0$, then $f$ is either in the second, fourth or fifth orbit in Proposition \ref{A2 orbits}. In the orbit $(x,y,xz,yz)$ the curvature locus is a plane and we take any orthonormal frame as an adapted frame. 
    \item[iii)] For the orbit $(x,y,xz,0)$ we take as $v_a^1$ the direction in which $\Delta_p$ is not bounded, i.e. the direction of the strip. 
    \item[iv)] Finally, for $(x,y,0,0)$, if $\Delta_p$ is a non-degenerate ellipse, take as $v_a^1$ and $v_a^2$ the unitary vectors in the directions of the semi-major and semi-minor axes, respectively, such that $\{v_a^1,v_a^2\}$ is a positively oriented orthonormal frame of $Ax_p$. If $\Delta_p$ is a segment, take as $v_a^1$ the direction of the segment and complete to obtain an orthonormal basis. If $\Delta_p$ is a point $y\neq p$ take $v_a^1$ such that $v_a^2=y/|y|$. If $y=p$ any orthonormal frame is an adapted frame.
\end{enumerate}

When $n=5$, $l=3$, so we have to define an adapted frame with 3 axial vectors. In \cite{BenediniRuasSacramento} there is a result similar to Proposition \ref{A2 orbits}.
\begin{prop}\label{A2 orbits5}(\cite{BenediniRuasSacramento})
There are six $\mathscr{A}^2$-orbits in $\Sigma^1J^2(3,5)$:
$$
(x,y,xz,yz,z^2),\ (x,y,z^2,yz,0),\ (x,y,xz,yz,0),\ (x,y,z^2,0,0),
$$ 
$$
 (x,y,xz,0,0)\ \mbox{and}\ (x,y,0,0,0).
$$
\end{prop}

With a discussion similar to Proposition \ref{types34} we can see what type of curvature locus there is in each orbit. The normal sections for corank 1 3-manifolds in $\mathbb R^5$ are corank 1 singular surfaces in $\mathbb R^4$, which have been studied in \cite{Benedini/Sinha/Ruas}. 

\begin{enumerate}
    \item [i)] For the orbit $(x,y,xz,yz,z^2)$, the primary axial vector $v_a^1$ can be defined as in Definition \ref{primaxialvect}. By Lemma \ref{change} the parametrisation of the first two components of the curvature locus can be taken to $a^1_{200}x^2 + 2a^1_{110}xy + a^1_{020}y^2+2a^1_{101}xz+2a^1_{011}yz$ and $a^2_{200}x^2 + 2a^2_{110}xy + a^2_{020}y^2+2a^2_{101}xz+2a^2_{011}yz$, where $x^2+y^2=1$ and $z\in\mathbb R$. Since in this orbit $a^1_{101}\neq0\neq a^2_{011}$, these are two unbounded functions so we can choose any orthonormal basis of this plane to complete the adapted frame $\{v_a^1,v_a^2,v_a^3\}$.
    \item[ii)] For the orbit $(x,y,z^2,yz,0)$, $v_a^1$ can be defined as in Definition \ref{primaxialvect}. By item b) in Proposition 4.13 in \cite{BenediniRuasSacramento} there is a direction perpendicular to $v_a^1$ such that the curvature locus is unbounded in both directions and a direction in which it is bounded. Choose this unbounded direction as $v_a^2$. We complete with a third vector to obtain our orthonormal adapted frame.
    \item[iii)] For the orbit $(x,y,xz,yz,0)$, the curvature locus is unbounded in a plane, so we choose any orthonormal frame to be $\{v_a^1,v_a^2\}$ and complete in a unique way to obtain $v_a^3$. Notice that in the direction of $v_a^3$ the curvature locus is bounded.
    \item[iv)] For the orbit $(x,y,z^2,0,0)$, $v_a^1$ can be defined as in Definition \ref{primaxialvect}. By Lemma \ref{change} (or Proposition 4.13 in \cite{BenediniRuasSacramento}), any germ in this orbit can be taken by changes of coordinates in the source and rotations in the target to the form $(x,y,\frac12(a^1_{200}x^2 + 2a^1_{110}xy + a^1_{020}y^2+a^1_{002}z^2),\frac12(a^2_{200}x^2 + 2a^2_{110}xy + a^2_{020}y^2),\frac12(a^3_{200}x^2 + 2a^3_{110}xy + a^3_{020}y^2))$. The curvature locus is given by $(a^1_{200}x^2 + 2a^1_{110}xy + a^1_{020}y^2+a^1_{002}z^2,a^2_{200}x^2 + 2a^2_{110}xy + a^2_{020}y^2,a^3_{200}x^2 + 2a^3_{110}xy + a^3_{020}y^2)$ where $x^2+y^2=1$. For $z=0$ we get the curvature ellipse (maybe degenerate) of the regular surface given by $f(x,y,0)$. For any other constant $z_0$, $II(C_q\cap\{z=z_0\})$ is the same curvature ellipse translated by $(z_0,0,0)$, i.e. a translation in the direction of the primary axial vector. This means that the curvature locus is a half-strip contained in a plane. Choose $v_a^2$ to be the orthogonal vector to $v_a^1$ in this plane. Finally, $v_a^3$ is the vector perpendicular to this plane. When the strip degenerates to a half-line choose the plane that contains the locus and the origin in order to define $v_a^2$, and $v_a^3$ follows as above. If the origin lies in the line that contains the half-line, choose any plane that contains the curvature locus and proceed as above.
    \item[v)] For the orbit $(x,y,xz,0,0)$, take $v_a^1$ as the direction in which the curvature locus is unbounded. Arguing as for the previous orbit, the curvature locus is a strip contained in a plane unbounded in the direction of $v_a^1$. Choose $v_a^2,v_a^3$ as in the previous case.
    \item[vi)] For the orbit $(x,y,0,0,0)$, the curvature locus can be any curvature locus of a regular surface in $\mathbb R^5$ (an ellipse, a segment or a point). Choose $v_a^1$ to be the semi-major axis if it is an ellipse, the direction of the segment if it is a segment, or any direction perpendicular to the point otherwise. Consider the plane that contains the ellipse, the plane that contains the locus and the origin if it is a segment, or any plane that contains $v_a^1$, the locus and the origin if it is a point, and define $v_a^2,v_a^3$ as in the previous two orbits. If the line that contains the segment contains the origin, choose any plane that contains the segment. If the point is the origin, choose any plane.
\end{enumerate}

\begin{ex}\label{ex3var5}
Similarly to Example \ref{ex.3var4}, we can illustrate some types of curvature loci for corank 1 3-manifolds in $\mathbb{R}^5$.
Let $M^{3}_{\sing}\subset\mathbb R^5$ be locally parametrised by $f:(\mathbb{R}^3,0)\rightarrow(\mathbb{R}^5,0)$.
\begin{itemize}
    \item[i)] Let $f(x,y,z)=\frac{1}{2}(x,y,x^2+z^2,xy+xz,3x^2+y^2+yz)$, whose 2-jet is $\mathscr{A}^2$-equivalent to $(x,y,xz,yz,z^2)$. The curvature locus is showed in Figure \ref{Fig3varR5i};
    \item[ii)] Consider $f(x,y,z)=\frac{1}{2}(x,y,xz,yz,x^2+3xy+y^2)$. In this case, $j^{2}f(0)$ is $\mathscr{A}^2$-equivalent to $(x,y,xz,yz,0)$ and the curvature locus is showed in Figure \ref{Fig3varR5ii};
    \item[iii)] Finally, consider $f(x,y,z)=\frac{1}{2}(x,y,3x^2+2xy+y^2+z^2,2x^2+5y^2,x^2+2y^2)$, whose 2-jet is $\mathscr{A}^2$-equivalent to $(x,y,z^2,0,0)$. In this case, Figure \ref{Fig3varR5iii} illustrates the curvature locus, a half-strip. 
\end{itemize}

\begin{figure}[ht]
\centering
\includegraphics[height=3.5cm,width=10cm]{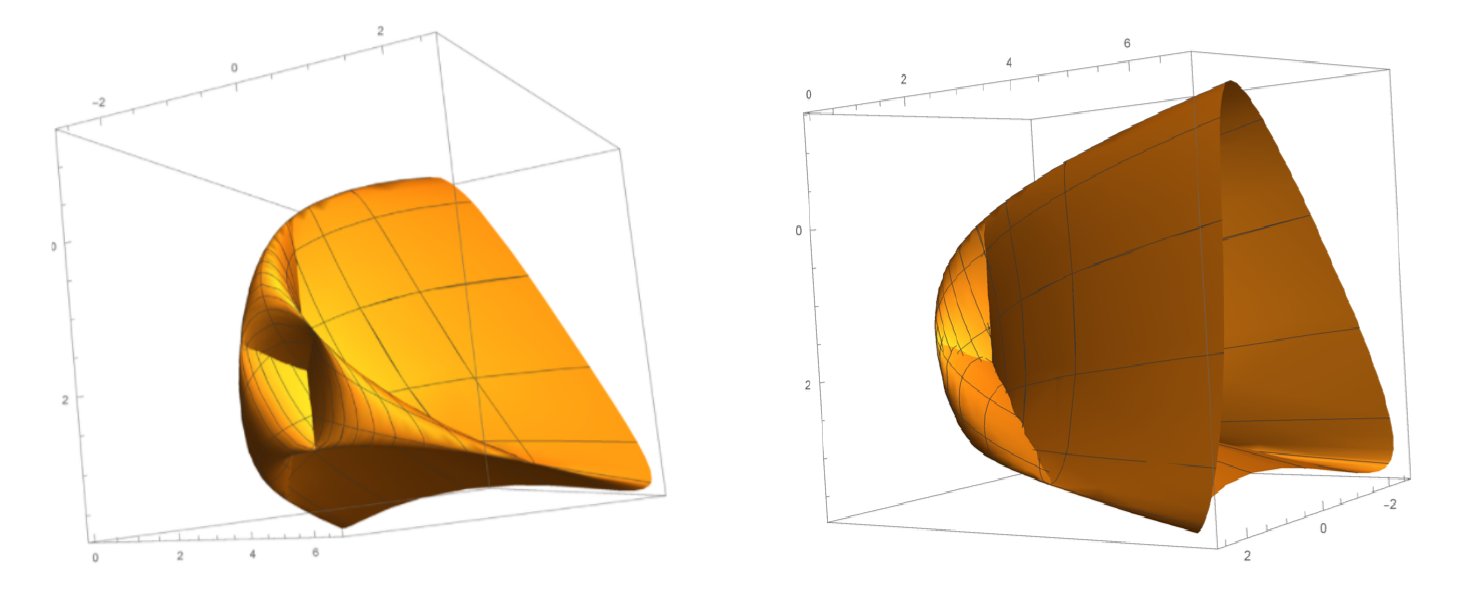}
\caption{Different views of $\Delta_p$ in Example \ref{ex3var5} (i).}
\label{Fig3varR5i}
\end{figure}

\begin{figure}[ht]
\centering
\includegraphics[height=3.5cm,width=10cm]{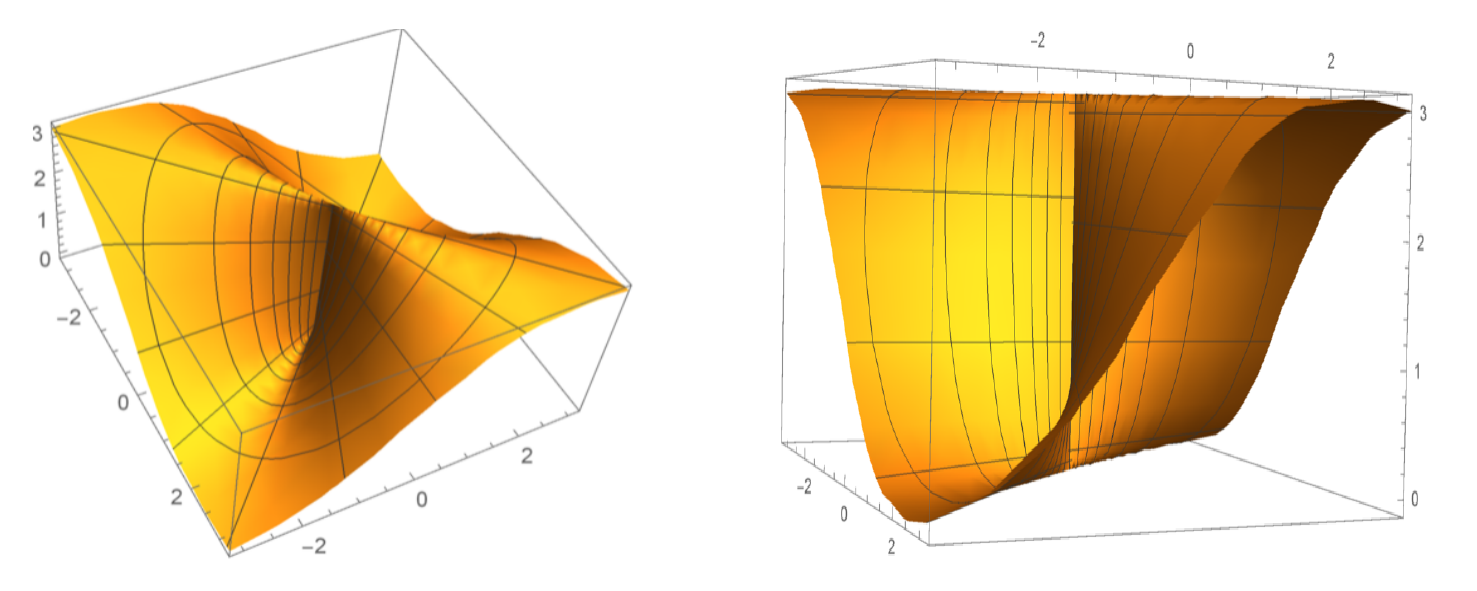}
\caption{Different views of $\Delta_p$ in Example \ref{ex3var5} (ii).}
\label{Fig3varR5ii}
\end{figure}

\begin{figure}[ht]
\centering
\includegraphics[height=3.5cm,width=10cm]{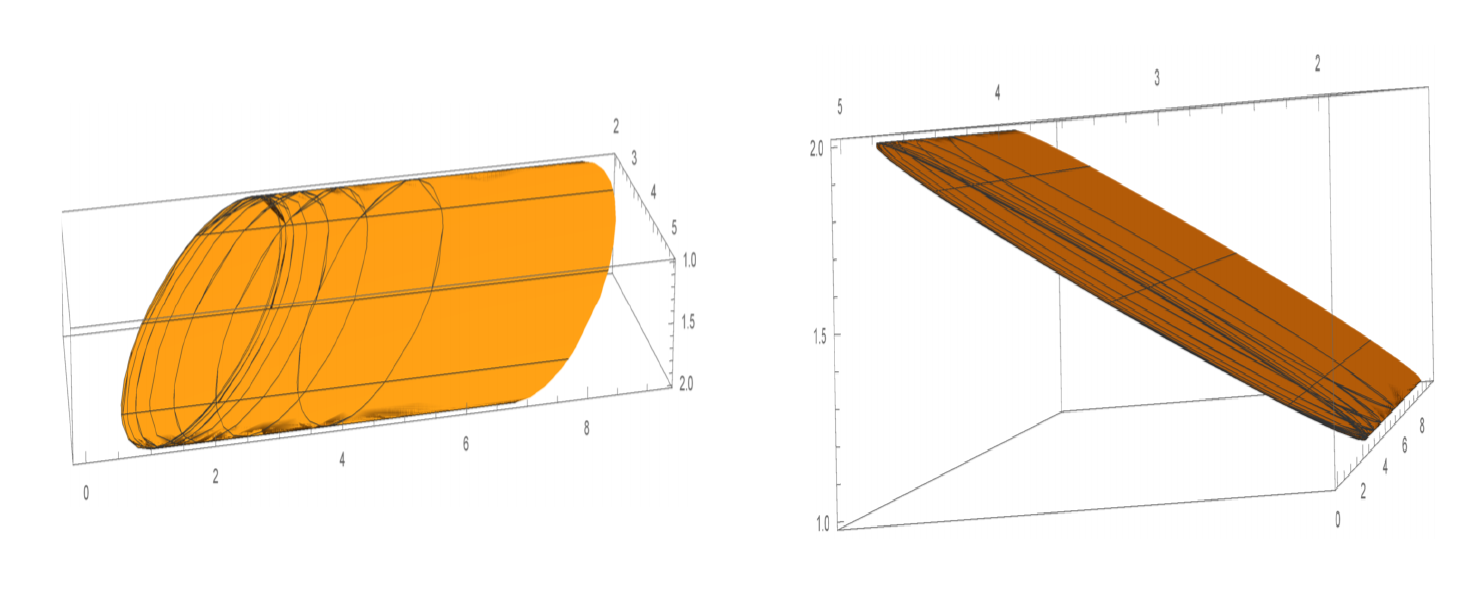}
\caption{Different views of $\Delta_p$ in Example \ref{ex3var5} (iii).}
\label{Fig3varR5iii}
\end{figure}







\end{ex}



\subsection{Geometrical interpretations}

\begin{teo}\label{axialprincipal}
Let $M^{3}_{\sing}\subset\mathbb R^{3+k}$ be given in Monge form (\ref{Monge3man}). 
Consider the regular surface $M^{2}_{\reg}\subset\mathbb R^{3+k}$ given by $f(x,y,0)$. If $II(u_{\infty},u_{\infty})=|| {\bf a_{002}}||\neq 0$, then the primary axial curvatures at $p\in M^{3}_{\sing}$ coincide with the $v_a^1$-principal curvatures at $p\in M^2_{\reg}$.
\end{teo}
\begin{proof}
Following the discussion at the end of Section 3, $T_qM^2_{\reg}$ is identified with $T_q\tilde M\cap \{z=0\}=u_{\infty}^{\perp}$  and the unitary tangent vectors are given by $C_q\cap u_{\infty}^{\perp}=\{(x,y)\in T_qM^2_{\reg}:x^2+y^2=1\}\equiv \mathbb{S}^1$. Therefore, the curvature ellipse of $M^{2}_{\reg}\subset\mathbb R^{3+k}$ at $p$ is contained in the curvature locus of $M^{3}_{\sing}\subset\mathbb R^{3+k}$ at $p$.

Given the primary normal curvature function $K_{v_a^1}(w)=\langle II(w,w),v_a^1 \rangle$ seen as a function from $\mathbb R^3$ to $\mathbb R$, the primary axial curvatures are given by the critical values of $K_{v_a^1}|_{C_q}$. The critical points of this function are given by the $2\times2$ minors of the following matrix
\begin{equation*}
\left(
\begin{array}{ccc}
\frac{\partial K_{v_a^1}}{\partial x} & \frac{\partial K_{v_a^1}}{\partial y} & \frac{\partial K_{v_a^1}}{\partial z} \\
2x & 2y & 0 \\
\end{array}
\right),
\end{equation*}
where in the second row we have the gradient of the equation of $C_q$, i.e. $x^2+y^2=1$. So the primary axial curvatures are given by the solutions of the system $2(\frac{\partial K_{v_a^1}}{\partial x}y-\frac{\partial K_{v_a^1}}{\partial y}x)=0, 2x\frac{\partial K_{v_a^1}}{\partial z}=0$ and $2y\frac{\partial K_{v_a^1}}{\partial z}=0$.

On the other hand, the eigenvalues of the shape operator $A_{v_a^1}:T_qM^2_{\reg}\to T_qM^2_{\reg}$ are given by the critical values of $\langle A_{v_a^1}(w),w\rangle=\langle II(w,w),v_a^1 \rangle$ where $w\in \mathbb{S}^1\subset T_qM^2_{\reg}$ are the unitary tangent vectors. So these are the critical values of $K_{v_a^1}|_{\mathbb{S}^1}=K_{v_a^1}|_{C_q\cap u_{\infty}^{\perp}}$, which are given by the determinant of
\begin{equation*}
\left(
\begin{array}{ccc}
\frac{\partial K_{v_a^1}}{\partial x} & \frac{\partial K_{v_a^1}}{\partial y} & \frac{\partial K_{v_a^1}}{\partial z} \\
2x & 2y & 0 \\
0 & 0 & 1
\end{array}
\right),
\end{equation*}
where the last row is the gradient of the equation for $C_q\cap u_{\infty}^{\perp}$, i.e. $z=0$. Therefore, the $v_a^1$-principal curvatures of $M^{2}_{\reg}\subset\mathbb R^{3+k}$ at $p$ are given by the solutions to $2(\frac{\partial K_{v_a^1}}{\partial x}y-\frac{\partial K_{v_a^1}}{\partial y}x)=0$.

Now, if $II(u_{\infty},u_{\infty})\neq 0$, then $|| {\bf a_{002}}||\neq0$, so we can use the normal form in Lemma \ref{change}. Now $v_a^1=(0,\ldots,0,1)$ and $K_{v_a^1}(w)=\bar{a}^{k+1}_{200}x^2 + 2\bar{a}^{k+1}_{110}xy + \bar{a}^{k+1}_{020}y^2 + z^2$ with $x^2+y^2=1$, so $\frac{\partial K_{v_a^1}}{\partial z}=2z$ and this is 0 if and only if $z=0$.

In conclusion, we have that the critical points of $K_{v_a^1}|_{C_q}$ coincide with the critical points of $K_{v_a^1}|_{C_q\cap u_{\infty}^{\perp}}$.
\end{proof}

\begin{teo}\label{axialprincipal2}
Let $M^{3}_{\sing}\subset\mathbb R^{3+k}$, $k=1,2$ and suppose that $f$ lies in the orbit $(x,y,z^2,0)$ or $(x,y,z^2,0,0)$ (depending on $k=1$ or $2$), then the secondary axial curvatures at $p\in M^{3}_{\sing}$ coincide with the $v_a^2$-principal curvatures at $p\in M^2_{\reg}$.
\end{teo}
\begin{proof}
Consider the Monge form as in Lemma \ref{change}. We prove it for $k=1$, the proof for $k=2$ is analogous. In the orbit $(x,y,z^2,0)$, the curvature locus is given by $2(a^1_{200}x^2 + a^1_{110}xy + a^1_{020}y^2,a^2_{200}x^2 + a^2_{110}xy + a^2_{020}y^2+z^2)$ where $x^2+y^2=1$. For $z=0$ we get the curvature ellipse (maybe degenerate) of the regular surface given by $f(x,y,0)$. For any other constant $z_0$, $II(C_q\cap\{z=z_0\})$ is the same curvature ellipse translated by $(z_0,0)$, i.e. a translation in the direction of the primary axial vector. Therefore the critical values of $K_{v_a^2}|_{C_q}$ coincide with the critical values of $K_{v_a^2}|_{C_q\cap u_{\infty}^{\perp}}$.
\end{proof}

\begin{rem}
When $k>2$, similarly to Proposition \ref{A2 orbits} there is an orbit of type $(x,y,z^2,0,\ldots,0)$. Here the secondary axial curvatures coincide with the $v_a^2$-principal curvatures. Notice also that in this case there may exist 3-ary axial curvatures, but the tangent space of $M^2_{\reg}$ is 2-dimensional.
\end{rem}

\begin{teo}\label{axialprincipal3}
Let $M^{3}_{\sing}\subset\mathbb R^{3+k}$, $k\geq 1$ and suppose that $f$ lies in the orbit $(x,y,0,\ldots,0)$, then the primary and secondary axial curvatures at $p\in M^{3}_{\sing}$ coincide with the $v_a^1$ and $v_a^2$-principal curvatures at $p\in M^2_{\reg}$ (resp.).
\end{teo}
\begin{proof}
In this orbit the curvature locus of the 3-manifold is precisely the curvature locus of the regular manifold so the result follows. In this case, with any other choice of adapted frame the result would still hold.
\end{proof}

\begin{rem}
Similar results to Theorems \ref{axialprincipal}, \ref{axialprincipal2} and \ref{axialprincipal3} can be proven for singular corank 1 $n$-manifolds in general. In fact, we believe that when $i\leq n-1$, then, in the orbits where $n-1$ $i$-ary axial curvatures are finite they coincide with the $v_a^i$-principal curvatures of the associated regular $(n-1)$-manifold. However, our proofs depend on the type of curvature locus and certain normal forms, so we do not have a general proof at the moment.
\end{rem}

\begin{ex}
Consider $f(x,y,z)=(x,y,\frac{3}{2}x^2+xy+\frac{1}{2}y^2+\frac{1}{2}yz,x^2+\frac{5}{2}y^2+\frac{1}{2}xz)$, which lies in the orbit $(x,y,xz,yz)$. Here the curvature locus is the whole plane and we can choose the adapted frame given by $v_a^1=(1,0)$ and $v_a^2=(0,1)$. One can check that $K_{v_a^1}(\theta,\phi)$ has a critical point with critical value $\kappa_{a_1}=3$. However, the $v_a^1$-principal curvatures of $f(x,y,0)$ are given by $2\pm \sqrt{2}$ and do not coincide with $\kappa_{a_1}$.

It is possible to find an adapted frame such that at least one of the principal curvatures in the direction of one of the vectors of the frame of the regular manifold coincides with an axial curvature, but it seems unlikely to be able to obtain a general result as Theorem \ref{axialprincipal}. However, we have the following partial result.
\end{ex}

\begin{prop}\label{axialunbound}
Let $M^3_{\sing}\subset\mathbb R^{3+k}$, $k=1,2$, if the curvature locus is unbounded on both sides in the direction of a certain axial vector $v_a^i$, then there is exactly 1 axial curvature in that direction. In particular, if there is exactly 1 direction $v_a^i$ in which the curvature locus is unbounded on both sides and $f$ is given in Monge form such that $v_a^i$ is one of the coordinate axes, then the corresponding component can be taken to $$j^2f_i(x,y,z)=\frac12({a}^i_{200}x^2 + 2{a}^i_{110}xy + {a}^i_{020}y^2 + 2{a}^i_{101}xz),$$ and the unique axial curvature in that direction is given by $\kappa_{a_i}={a}^{i}_{020}$. Furthermore, in this case, if ${a}^{i}_{110}=0$ then it coincides with 1 $v_a^i$-principal curvature of the associated regular surface.
\end{prop}
\begin{proof}
First suppose there is a unique direction corresponding to $v_a^i$ such that the curvature locus is unbounded on both sides. This is the case of the orbits $(x,y,z^2,xz)$ and $(x,y,xz,0)$ when $k=1$ and the orbits $(x,y,z^2,xz,0)$ and $(x,y,xz,0,0)$ when $k=2$. By Proposition 4.13 in \cite{BenediniRuasSacramento}, for $k=2$ the component corresponding to $v_a^i$ can be taken, by changes of variable in the source and isometries in the target, to $$j^2f_i(x,y,z)=\frac12({a}^i_{200}x^2 + 2{a}^i_{110}xy + {a}^i_{020}y^2 + 2{a}^i_{101}xz).$$ The proof of Proposition 4.13 is valid for the two orbits in $k=1$ too. The normal curvature function in the unbounded direction is given by
$$
K_{v_{a}^i} (w) =\langle \eta(w), v_a^i\rangle={a}^{i}_{200}\cos^2\theta + 2{a}^{i}_{110}\cos\theta\sin\theta + {a}^{i}_{020}\sin^2\theta + \cos\theta\frac{\cos\phi}{\sin\phi}.
$$ 
Taking the partial with respect to $\phi$ equal to 0 we get $\cos\theta=0$ and so $\sin\theta=\pm 1$. Therefore the unique critical value is $\kappa_{a_i}={a}^{i}_{020}$. 

On the other hand, if ${a}^{i}_{110}=0$ the end points of the projection of the curvature ellipse of the associated regular surface in the direction $v_a^i$ (i.e. the $v_a^i$-principal curvatures) are given by ${a}^{i}_{200}$ and ${a}^{i}_{020}$.

Now suppose that there are more than one axial vectors in which the curvature locus is unbounded on both sides. This is the case of the orbits $(x,y,xz,yz)$ when $k=1$ or $(x,y,xz,yz,z^2)$ and $(x,y,xz,yz,0)$ when $k=2$. In the plane in which the curvature locus is unbounded we can chose any orthonormal frame to be part of the adapted frame. Take a vector $v_a^i=(\alpha,\beta)$ (if $k=1$, $(\alpha,\beta,0)$ if $k=2$) and, by Proposition 4.13 in \cite{BenediniRuasSacramento} (the proof is also valid for the orbit $(x,y,xz,yz)$), we can consider the normal curvature function
$
K_{v_{a}^i} (w) =\langle \eta(w), v_a^i\rangle=
$
$$
=\alpha({a}^{i}_{200}\cos^2\theta + 2{a}^{i}_{110}\cos\theta\sin\theta + {a}^{i}_{020}\sin^2\theta + \cos\theta\frac{\cos\phi}{\sin\phi}+2a_{011}^i\sin\theta\frac{\cos\phi}{\sin\phi})+$$
$$
+\beta({a}^{i+1}_{200}\cos^2\theta + 2{a}^{i+1}_{110}\cos\theta\sin\theta + {a}^{i+1}_{020}\sin^2\theta + 2a_{011}^{i+1}\sin\theta\frac{\cos\phi}{\sin\phi}).
$$
The partial derivative with respect to $\phi$ is 0 if and only if $\alpha\cos\theta+2(\alpha a_{011}^i+\beta a_{011}^{i+1})\sin\theta=\langle(\alpha,2(\alpha a_{011}^i+\beta a_{011}^{i+1})),(\cos\theta,\sin\theta)\rangle=0$, so there are two values of $\theta$ for which we may have critical points. Substituting $\alpha\cos\theta+2(\alpha a_{011}^i+\beta a_{011}^{i+1})\sin\theta=0$ in $K_{v_{a}^i} (w)$ we can see that the critical value does not depend on $\phi$. The critical points are two lines of antipodal points in the cylinder $C_q$. Since the second fundamental form is quadratic homogeneous, the image of antipodal points is the same, so the image of the two lines is the same and there is only 1 critical value.
\end{proof}

\begin{coro}
For $M^{3}_{\sing}\subset\mathbb R^{3+k}$ there are at least 2 and at most 4 axial curvatures when $k=1$ and at least 4 and at most 5 axial curvatures when $k=2$, which may coincide in degenerate cases.
\end{coro}
\begin{proof}
From Proposition \ref{numaxial}, $l(n-1)$ is a higher bound for the number of axial curvatures. 

For $k=1$, $l=2$ and so there are at most 4 axial curvatures. This higher bound is attained in orbits such as $(x,y,z^2,0)$ or $(x,y,0,0)$, where the axial curvatures coincide with principal curvatures of the associated regular surfaces (by Theorems \ref{axialprincipal}, \ref{axialprincipal2}, and \ref{axialprincipal3}). On the other hand, there are at most 2 directions in which the curvature locus is unbounded on both sides (i.e. in the orbit $(x,y,xz,yz)$) and by Proposition \ref{axialunbound}, there will be exactly 1 axial curvature in each.

For $k=2$, $l=3$. However, the higher bound 6 is not attained since by the way the adapted frame is chosen, whenever there are 2 primary and 2 secondary axial curvatures, there is only 1 3-ary axial curvature ($v_a^3$ is perpendicular to the plane that contains the curvature locus, and so the projection to this direction gives only 1 value). On the other hand, there are at most two directions in which the curvature locus is unbounded on both sides of the direction of the axial vector (in the orbits $(x,y,xz,yz,z^2)$ and $(x,y,xz,yz,0)$) by Proposition \ref{axialunbound} there is only 1 axial curvature in each of these directions. In the remaining direction there will be 2 axial curvatures.

In degenerate cases, when the curvature locus is a segment, for example, two $i$-ary curvatures might coincide. In this case, there might be only 3 different axial curvatures when $k=2$.
\end{proof}

As corollaries of Theorems \ref{axialprincipal}, \ref{axialprincipal2}, \ref{axialprincipal3} and Proposition \ref{axialunbound} we get some interesting geometrical interpretations.

\begin{coro}\label{GaussianCurva}
Let $M^{3}_{\sing}\subset\mathbb R^{4}$ and suppose that $f$ lies in the orbits $(x,y,z^2,0)$ or $(x,y,0,0)$, the Gaussian curvature of the regular surface given by $f(x,y,0)$ is given by 
$$K=\kappa_{a_1}^1\kappa_{a_1}^2+\kappa_{a_2}^1\kappa_{a_2}^2.$$
In particular, this includes when $f$ is a frontal and  $f(x,y,0)$ is a cuspidal surface or a cuspidal surface with a curve of cuspidal cross-caps.
\end{coro}
\begin{proof}
Given a regular surface in $\mathbb R^4$ and $\{v_1,v_2\}$ an orthonormal basis of the normal plane, by Theorem 1 in \cite{BastoGoncalves} (which can be found as Theorem 7.1 in \cite{Livro}), the Gaussian curvature is $K=K_1+K_2$, where $K_i$ is the Gaussian curvature of the regular surface in $\mathbb R^3$ given by the projection in the normal direction orthogonal to $v_i$. Given the adapted frame $\{v_a^1,v_a^2\}$ of $N_pM$, then $K_i$ is the product of the $v_a^i$-principal curvatures, $i=1,2$. Therefore, by Theorems \ref{axialprincipal}, \ref{axialprincipal2} and \ref{axialprincipal3}, we get $K=\kappa_{a_1}^1\kappa_{a_1}^2+\kappa_{a_2}^1\kappa_{a_2}^2.$
\end{proof}


\begin{coro}
Let $M^{3}_{\sing}\subset\mathbb R^{4}$ and suppose that $f$ is such that the normal curvature tensor $R_D(p)=0$, then the curvature locus of the associated $M^2_{\reg}$ is a segment with extremal points $P_1=(\kappa_{a_1}^1,\kappa_{a_2}^1)$ and $P_2=(\kappa_{a_1}^2,\kappa_{a_2}^2)$.
\end{coro}
\begin{proof}
For a regular $n$-manifold in $\mathbb R^{n+2}$, when $R_D(p)=0$ there is a unique orthonormal basis of eigenvectors $\{X_1,\ldots,X_n\}$ of the shape operator $A_{v}$ for all $v$ normal vectors. In \cite{NunoRomero} it is shown that, in this case the curvature locus is an $n$-polygon (a convex polygon with $n$ sides) such that the vertices are given by $(\lambda_i,\mu_i)$ where $\lambda_i$ and $\mu_i$ are the eigenvalues corresponding to $\{X_1,\ldots,X_n\}$ for $v_1$ and $v_2$, respectively. The result follows from Theorems \ref{axialprincipal} and \ref{axialprincipal2}
\end{proof}

\begin{coro}
Let $M^{3}_{\sing}\subset\mathbb R^{4}$ and suppose that $f$ lies in the orbits $(x,y,xz,z^2)$ or $(x,y,xz,0)$. If $f$ is given in Monge form and $a_{020}^2=0$,then the curvature of the curve $f(0,y,0)$ is equal to $\kappa_{a_2}$ if $f$ lies in the first orbit and equal to $\kappa_{a_1}$ otherwise. In particular, this is the curvature of the curve of cross-caps when $f$ is in the orbit $(x,y,xz,z^2)$ and it is the curvature of the curve of swallowtail points for the corresponding example in the orbit $(x,y,xz,0)$.
\end{coro}
\begin{proof}
The 2-jet of the curve $f(0,y,0)$ is given by $(0,y,a_{020}^1y^2,0)$. The result now follows by Proposition \ref{axialunbound}.
\end{proof}

The following result gives a formula to calculate the axial curvatures.

\begin{prop}\label{formula-M3}
Let $M^3_{\sing}\subset\mathbb R^{3+k}$ be given by the image of $f$ in Monge form as in Lemma \ref{change}. Suppose $\kappa^1_{a_1}(p)$ and $\kappa^2_{a_1}(p)$ are defined (i.e. finite), then:
\begin{enumerate}
    \item When $(\bar{a}^{k+1}_{020}-\bar{a}^{k+1}_{200})\neq0$. If $\bar{a}^{k+1}_{110}\neq0$, there exist two  primary axial curvatures given by
    $$
    \kappa^i_{a_1}(p)=K_{v^1_{a}}\left(\theta_i,\frac{\pi}{2}\right),  
    $$
    where $\left(\theta_i,\frac{\pi}{2}\right)$ are two critical points of 
    $$
    K_{v^1_a} (w) =\langle \eta(w), v_a\rangle=\bar{a}^{k+1}_{200}\cos^2\theta + \bar{a}^{k+1}_{110}\sin(2\theta) + \bar{a}^{k+1}_{020}\sin^2\theta + \cot^2{\phi}.
    $$ 
    In particular, if $\bar{a}^{k+1}_{110}=0$, then
    $$
    \kappa^1_{a_1}(p)=\bar{a}^{k+1}_{200} \mbox{ and $\kappa^2_{a_1}(p)=\bar{a}^{k+1}_{020}$}.
    $$
    \item  When $\bar{a}^{k+1}_{020}=\bar{a}^{k+1}_{200}$, 
    $$
    \kappa^1_{a_1}(p)=\kappa^2_{a_1}(p)=\bar{a}^{k+1}_{200}-|\bar{a}^{k+1}_{110}|.
    $$
    
  
    \end{enumerate}
\end{prop}
\begin{proof}
By Lemma \ref{change}, we can consider the primary axial vector as $v_{a_1}=(0,\ldots,0,1)$ and 
$$
\eta(w)=(*,\ldots,*,\bar{a}^{k+1}_{200}\alpha^2 + 2\bar{a}^{k+1}_{110}\alpha\beta + \bar{a}^{k+1}_{020}\beta^2 + \gamma^2),
$$ 
where $w=(\alpha,\beta,\gamma)\in C_q$. Then the axial normal curvature at the direction $v_{a_1}$ is given by 
$$
K_{v_{a}^1} (w) =\langle \eta(w), v_{a}^1\rangle=\bar{a}^{k+1}_{200}\alpha^2 + 2\bar{a}^{k+1}_{110}\alpha\beta + \bar{a}^{k+1}_{020}\beta^2 + \gamma^2
$$ 
As $\alpha^2+\beta^2=1$, we can take  $\alpha=\cos\theta$, $\beta=\sin\theta$, $\gamma=\mathrm{cotan} \phi$, then  
$$
K_{v_{a}^1} (w) =\langle \eta(w), v_a^1\rangle=\bar{a}^{k+1}_{200}\cos^2\theta + 2\bar{a}^{k+1}_{110}\cos\theta\sin\theta + \bar{a}^{k+1}_{020}\sin^2\theta + \mathrm{cotan}^2 \phi
$$ 
The critical points of $K_{v_{a}^1}$ are the points $(\theta,\frac{\pi}{2})$ such that $\theta$ satisfies the equation
$$
(\bar{a}^{k+1}_{020}-\bar{a}^{k+1}_{200})\sin(2\theta)+2\bar{a}^{k+1}_{110}\cos(2\theta)=0.
$$
We have the following cases:
\begin{enumerate}
    \item When $(\bar{a}^{k+1}_{020}-\bar{a}^{k+1}_{200})\neq0$, then $\sin(2\theta)=\frac{2\bar{a}^{k+1}_{110}}{(\bar{a}^{k+1}_{200}-\bar{a}^{k+1}_{020})}\cos(2\theta)$. Thus, if $\cos(2\theta)\neq0$, we obtain  $\theta=\frac12\arctan\left(\frac{2\bar{a}^{k+1}_{110}}{\bar{a}^{k+1}_{200}-\bar{a}^{k+1}_{020}}\right)$. Note that there are four possibles values of $\theta$, which we call $\theta_i$, $i=1,2,3,4$, such that $K_{v_{a}^1}(\theta_i,\pi/2)=\bar{a}^{k+1}_{020}+(\bar{a}^{k+1}_{200}-\bar{a}^{k+1}_{020})\cos^2(\theta_i))+\bar{a}^{k+1}_{110}\sin(2\theta_i))$ is a critical value, but there are only two critical values. 
    Here the axial curvatures are given by
    $$
    \kappa_{a_1}^j=\{K_{v_{a}^1}\left(\theta_i,\frac{\pi}{2}\right)| \mbox{ for $i=1,2,3,4$}\}.
    $$
    In particular, $\bar{a}^{k+1}_{110}=0$ iff $\sin(2\theta)=0$, then $(0,\frac{\pi}{2}),(\pi,\frac{\pi}{2})$ and $(\frac{\pi}{2},\frac{\pi}{2}) ,(\frac{3\pi}{2},\frac{\pi}{2})$ 
    are the critical points, so
    $$
    \kappa^1_{a_1}(p)=\bar{a}^{k+1}_{200} \mbox{ and $\kappa^2_{a_1}(p)=\bar{a}^{k+1}_{020}$}.
   $$
   
   
    \item If $\bar{a}^{k+1}_{020}=\bar{a}^{k+1}_{200}$ and  $\bar{a}^{k+1}_{110}\neq0$, we get 
    $\cos(2\theta)=0$. 
    Thus, 
    if $\bar{a}^{k+1}_{110}>0$,  $(\frac{3\pi}{4},\frac{\pi}{2})$ and $(\frac{7\pi}{4},\frac{\pi}{2})$ are critical points and if $\bar{a}^{k+1}_{110}<0$, the critical points are $(\frac{\pi}{4},\frac{\pi}{2})$ and $(\frac{5\pi}{4},\frac{\pi}{2})$, in both cases  
    $$
    \kappa_{a_1}^1(p)=\kappa^2_{a_1}(p)=\bar{a}^{k+1}_{200}-|\bar{a}^{k+1}_{110}|.
    $$
      In particular, when $\bar{a}^{k+1}_{110}=0$, $(\theta,\pi/2)$ are critical points of $K_{v_{a}^1}$ for all $\theta\in\left[0,2\pi\right)$ and the axial curvature is $\kappa_{a_1}^1=\kappa_{a_1}^2=\bar{a}^{k+1}_{020}=\bar{a}^{k+1}_{200}$.
\end{enumerate}
\end{proof}

\begin{coro}\label{ka2}
Let $M^3_{\sing}\subset\mathbb R^{4}$ be given by the image of $f$ in Monge form as in Lemma \ref{change} when $rank(A)=1$, i.e. $f$ lies in the orbit $(x,y,z^2,0)$, 
then: 
\begin{enumerate}
    \item When $(\bar{a}^{1}_{020}-\bar{a}^{1}_{200})\neq0$.  If $\bar{a}^{1}_{110}\neq0$, there exist at most two secondary axial curvatures given by
    $$
    \kappa^i_{a_2}(p)=K_{v^2_{a}}\left(\theta_i\right),  
    $$
    where $\theta_i$ are critical points of 
    $$
    K_{v^2_a} (w) =\langle \eta(w), v^2_a\rangle=-(\bar{a}^{1}_{200}\cos^2\theta + \bar{a}^{1}_{110}\sin(2\theta) + \bar{a}^{1}_{020}\sin^2\theta).
    $$ 
    In particular, if $\bar{a}^{1}_{110}=0$, then
    $$
    \kappa^1_{a_2}(p)=-\bar{a}^{1}_{200} \mbox{ and $\kappa^2_{a_2}(p)=-\bar{a}^{1}_{020}$}.
    $$
    \item  When $\bar{a}^{1}_{020}=\bar{a}^{1}_{200}$, 
    $$
    \kappa^1_{a_2}(p)=\kappa^2_{a_2}(p)=-(\bar{a}^{1}_{200}-|\bar{a}^{1}_{110}|).
    $$
\end{enumerate}
\end{coro}
\begin{proof}
The proof follows from the fact that $v^1_{a}=(0,1)$ and $v_a^2=(0,-1)$ and analogous arguments to the proof of Proposition \ref{formula-M3}.
\end{proof}

\begin{rem}
The obstruction to generalizing the above result for $M^3_{\sing}\subset\mathbb R^{3+k}$ with $k>1$ is the fact that, although with the normal form of Lemma \ref{change} $v_a^1=(0,\ldots,0,1)$, the secondary axial vector may not coincide with one of the axes in that coordinate system.
\end{rem}

\begin{ex} 
\begin{enumerate}
    \item[i)] Consider the frontal 3-manifold given by 
 $(x,y,\frac{1}{2}(a^1_{200}x^2 + a^1_{020}y^2+z^2),\frac{1}{2}(a^2_{200}x^2 + a^2_{020}y^2+yz^3)).$ Here $(x,0,\frac{1}{2}a^1_{200}x^2,\frac{1}{2}a^2_{200}x^2)$ is a curve of cuspidal cross-cap points. Let $\kappa$ be the curvature of this curve, then, by Proposition \ref{formula-M3} and Corollary \ref{ka2} $\kappa^2=(\kappa_{a_1}^1)^2+(\kappa_{a_2}^1)^2$.
   \item[ii)] Consider the frontal 3-manifold given by 
 $(x,y,\frac{1}{2}(a^1_{200}x^2 + a^1_{110}xy + a^1_{020}y^2+2xz+4z^3),\frac{1}{2}(a^2_{200}x^2 + a^2_{110}xy + xz^2+3z^4)).$ Here $(0,y,\frac{1}{2}a^1_{020}y^2,0)$ is a curve of swallowtail points and by Proposition \ref{axialunbound} the curvature of this curve is given by $\kappa_{a_1}$.
\end{enumerate}
\end{ex}

\begin{ex}
Consider $M^3_{\sing}\subset\mathbb R^4$ given by 
$$
f(x, y, z) = (x, y, \frac12x^2+\frac72y^2,\frac32x^2 +xy+\frac12y^2+\frac12z^2)
$$
Observe that $f$ lies in the orbit $(x,y,z^2,0)$ and the primary axial vector is $v^1_a=(0,1)$. Direct calculation shows $(\frac\pi8,\frac\pi2),(\frac{5\pi}{8},\frac\pi2)$ are critical points of $K_{v^1_a}$ and $\kappa^i_{a_1}(p)=1+\sin(2\theta_i)+2\cos^2(\theta_i)$, so for $\theta=\frac\pi8,\frac{5\pi}{8}$  $\kappa^1_{a_{1}}(p)=2+\sqrt{2}$ and $\kappa^2_{a_{1}}(p)=2-\sqrt{2}$, respectively. For the secondary axial vector $v^2_a=(-1,0)$, we get $\theta=0,\frac\pi2$ are critical points of $K_{v^2_a}$. So $\kappa^1_{a_{2}}(p)=-1$ and $\kappa^2_{a_{2}}(p)=-7$.
It follows from Corollary \ref{GaussianCurva} that the Gaussian curvature of the regular surface $f(x,y,0)$ is $$K=\kappa_{a_1}^1\kappa_{a_1}^2+\kappa_{a_2}^1\kappa_{a_2}^2=(2+\sqrt{2})\cdot(2-\sqrt{2})+(-1)\cdot(-7)=2+7=9.$$
\end{ex}

\begin{ex}
Consider $M^3_{\sing}\subset\mathbb R^5$ given by 
$$
f(x, y, z) = (x, y, x^2, y^2,\frac32x^2 +xy+\frac12y^2+\frac12z^2)
$$
The primary axial vector is $v^1_a=(0,0,1)$ and  $(\frac\pi8,\frac\pi2),(\frac{5\pi}{8},\frac\pi2)$ are the critical points of $K_{v^1_a}$. The axial curvatures at the origin are given by $\kappa^i_a(p)=1+\sin(2\theta)+2\cos^2(\theta)$, where $\theta=\frac\pi8,\frac{5\pi}{8}$.
Now, consider $u$ a tangent direction in $T_pM^3_{\sing}$ parametrised by the angle $\gamma\in\left[0,2\pi\right)$, 
then the normal section of $M^3_{\sing}$ along $u$ is a corank 1 surface $M^2_{\sing} = M^3_{\sing} \cap\{u=Y - \tan(\gamma)X = 0\}$ that is locally parametrised by
$$
(x, z)\to (x, \tan(\gamma)x, x^2, \tan^2(\gamma) x^2,
\frac12 \tan^2(\gamma)x^2 + \tan(\gamma) x^2 + \frac32 x^2 + \frac12z^2)
$$
By a rotation of angle $\gamma\neq\pi/2, 3\pi/2$ in the target and the change of coordinates in the source, $(x, z)\mapsto \left(\frac{x}{
\tan^2(\gamma)+1} , z\right)$ we obtain the $M^2_{\sing}\subset\mathbb R^4$ given by $(x, z)\mapsto$
$$ (x, 0,\frac{x^2}{\tan^2(\gamma)+1}, \frac{\tan^2(\gamma)}{\tan^2(\gamma)+1} x^2,
\\ \frac12\frac{(\tan^2(\gamma)+2\tan(\gamma)+3)x^2}{\tan^2(\gamma)+1}+\frac12z^2).
$$
Here, the curvature locus is $\eta(z)=(*,*,\frac{(\tan^2(\gamma)+2\tan(\gamma)+3)}{\tan^2(\gamma)+1}+z^2)$ and the primary axial vector for this singular surface is $v^1_a=(0,0,1)$. Then, using Proposition \ref{axialcurvatureR4}, the primary axial curvature of $M^2_{\sing}$ at the origin is $k_{a_1}(p)=\frac{(\tan^2(\gamma)+2\tan(\gamma)+3)}{\tan^2(\gamma)+1}=1+\sin(2\gamma)+2\cos^2(\gamma)$.
Note that, the primary axial curvature of $M^2_{\sing}$ coincides with  the primary axial curvature of $M^3_{\sing}$ when $\gamma=\theta$. 
\end{ex}
Based on the above example it is natural to ask whether there is a relation between the axial curvatures at $p\in M^3_{\sing}\subset\mathbb R^5$ and at
$p\in M^2_{\sing}\subset\mathbb R^4$ obtained as a normal section of $M^3_{\sing}$.

\begin{teo} Let $M^3_{\sing}$ be a corank 1 3-manifold in $\mathbb R^5$ and consider $M^2_{\sing}\subset\mathbb R^4$ the normal section of $M^3_{\sing}$ along $u\in T_pM^3_{\sing}$. The $i$-ary axial curvatures at $p\in M^3_{\sing}$ (whenever they are defined) are given by the critical values of the $i$-ary axial curvatures of the normal sections varying $u\in T_p M^3_{\sing}$.
\end{teo}
\begin{proof}
Following Theorem 3.3 in \cite{BenediniOset2} the curvature locus of $M^3_{\sing}$ is generated by the union of the curvature parabolas of the normal sections. This is due to the fact that the unitary tangent vectors $C_q'$ of the normal section are given by $C_q\cap dg^{-1}(u)$, where $g:\tilde M\to M^3_{\sing}$ is the corank 1 map germ and $C_q$ is the unitary tangent vectors in $T_q\tilde M$. This way the primary axial vector is the same for the 3-manifold and any normal section, and so we have the relation between the axial curvatures. In fact, the $i$-ary normal axial curvature function $K_{a_i}(\theta,\phi)$ at $\theta=\gamma_0$, where $\gamma_0$ is the angle which parametrises $u$, is precisely the $i$-ary normal curvature function of the normal section given by the angle $\gamma_0$, $K_{a_i}^{\gamma_0}(\phi)$. So the critical values of $K_{a_i}(\theta,\phi)$ are the critical values of the function $h:[0,2\pi[\to \mathbb R$ given by $h(\gamma)=(\text{critical value of   } K_{a_i}^{\gamma}(\phi)\text{   where   } \phi\in [0,2\pi[)$ when $\gamma$ varies in $[0,2\pi[$.
\end{proof}


\section{Relation of axial curvatures and umbilic curvatures}\label{umbs}

Throughout the literature the umbilic curvature has been defined in many different contexts, both for regular and singular manifolds. It was first defined by Montaldi in \cite{Montaldi}  for semi-umbilic points in $M^2_{\reg}\subset \mathbb R^4$. Then it was defined by Mochida, Romero Fuster and Ruas in \cite{MochidaFusterRuas3}  for $M^2_{\reg}\subset \mathbb R^5$. More recently, Deolindo-Silva and Oset Sinha in \cite{Deolindo/Oset}  defined it for $M^3_{\reg}\subset \mathbb R^6$ (in \cite{Binotto/Costa/Fuster}  Binotto, Costa and Romero Fuster mention focal spheres but do not define the umbilic curvature). It has also been defined for singular manifolds, namely, Martins and Nu\~no-Ballesteros in \cite{MartinsBallesteros}  defined it for $M^2_{\sing}\subset \mathbb R^3$ (it was also defined in the frontal context by Martins and Saji in \cite{MartinsSaji} where they called it the limiting normal curvature) and finally, Benedini Riul, Oset Sinha and Ruas for $M^2_{\sing}\subset \mathbb R^4$ in \cite{Benedini/Sinha/Ruas}. The name ``umbilic" is not casual, this curvature defines the centre of a sphere with degenerate contact with the manifold at $p$, also known as an umbilical focus and umbilical focal hyperspheres. For surfaces it is a sphere with corank 2 contact (i.e. corank 2 singularity of the distance squared function), for 3-manifolds it is a sphere with corank 3 contact. This was proved in the previous references when it was defined, except for the case $M^2_{\sing}\subset \mathbb R^4$, where it was proved by Deolindo-Silva and Oset Sinha in \cite{Deolindo/Oset}.

The umbilic curvature has been defined in many different ways but all definitions can be unified in the following way. Given $M^n\subset\mathbb R^{n+k}$ (regular or singular) the umbilic curvature is given by $\kappa_u(p)=d(p,Aff_p)$. If $\dim N_pM>n$ this definition makes sense always, if $\dim N_pM\leq n$, this is only defined when the curvature locus is degenerate. For example, Montaldi only defined this curvature when the curvature locus is a segment. Similarly, Martins and Nu\~no-Ballesteros also defined the umbilic curvature when the curvature locus is a degenerate parabola. However, Benedini Riul, Oset Sinha and Ruas defined it even for non-degenerate parabolas, following the ideas of Mochida, Romero Fuster and Ruas. Another way of defining this curvature is by projecting to a direction perpendicular to $Aff_p$ (since $d(p,Aff_p)$ is the shortest distance from $p$ to points in $Aff_p$).

Focusing now in the singular case, in most cases, the umbilic curvature is one of our axial curvatures.

\begin{prop}
If $n\geq \dim N_pM=k+1$, when $\kappa_u$ is defined, $\kappa_u=|\kappa_{a_i}|$ for some $i=2,\ldots,l$. More precisely, if $Aff_p$ has codimension $r<k+1$ in $N_pM$, then $\kappa_u=|\kappa_{a_{l-r+1}}|$. When $r=k+1$, the curvature locus is a point and $\kappa_u=|\kappa_{a_{2}}|$ by definition. 
\end{prop}
\begin{proof}
We have defined our adapted frame only in $Ax_p$, and so we only have $l$ axial vectors. If $n\geq \dim N_pM=k+1$, then $l=k+1$, and so $Ax_p=N_pM$. In this case, $\kappa_u$ is only defined when the curvature locus is degenerate, i.e. $Aff_p$ is not the whole $N_pM$. So $\kappa_u$ is the projection of the curvature locus onto a direction perpendicular to $Aff_p$. The way in which the adapted frame has been defined, this direction is one of the axial vectors. In fact, when $Aff_p$ is a hyperplane, this direction is $v_a^l$ and so $\kappa_u=|\kappa_{a_l}|$. If $Aff_p$ has codimension 2 in $N_pM$ then this direction is $v_a^{l-1}$, $\kappa_u=|\kappa_{a_{l-1}}|$ and $\kappa_{a_l}=0$. This goes on until codimension $k$. When $r=k+1$, the curvature locus is a point and $\kappa_u=|\kappa_{a_{2}}|$ by definition (there is no need for the absolute value here because $\kappa_{a_2}$ is defined as norm in this case). 
\end{proof}

If $n<\dim N_pM=k+1$ there are less axial vectors than the dimension of the normal space, and so the umbilic curvature may be given by a projection to a normal direction which is not in $Ax_p$. However, consider the $(k+1-l)$-vector space of orthogonal directions to the directions in $Ax_p$. The intersection of this vector space and $Ax_p$ is a point. Consider the direction between this point and $p$. We call the unitary vector in this direction $v_a^{l+1}$. Notice that the projection of the curvature locus onto this direction is constant. We call this constant $\kappa_{v_a^{l+1}}$. 

\begin{prop}
If $n<\dim N_pM=k+1$ then $\kappa_u=|\kappa_{v_a^{l+1}}|$ when $\dim Aff_p=l$ and $\kappa_u=|\kappa_{a_i}|$ for some $i=2,\ldots,l$ otherwise.
\end{prop}
\begin{proof}
In this case, $l=n$. When $\dim Aff_p< l$, then $Ax_p$ contains $Aff_p$ and $p$, so, similarly to the above Proposition, $\kappa_u=|\kappa_{a_i}|$ for some $i=2,\ldots,l$. When $\dim Aff_p=l$, then $Ax_p=Aff_p$ and so $d(p,Aff_p)$ coincides with $|\kappa_{v_a^{l+1}}|$ by definition. 
\end{proof}

\begin{ex}
In \cite{Benedini/Sinha/Ruas} the umbilic curvature $\kappa_u$ was defined for $M^2_{\sing}\subset \mathbb R^4$. Here $Ax_p$ is a plane with adapted frame $\{v_a^1,v_a^2\}$. When the curvature parabola is degenerate (a half-line, a line or a point) $\kappa_u=|\kappa_{a_2}|$. When the curvature parabola is non-degenerate, $\kappa_u$ is the height of $Aff_p=Ax_p$ and so $\kappa_u=|\kappa_{v_a^3}|$ as defined above. 
\end{ex}

\begin{rem}
In the same way as for all other situations where the umbilic curvature has been defined, a geometric interpretation of the umbilical curvature for $M^3_{\sing}\subset \mathbb R^5$ can be given. For example, when $\dim Aff_p=2$ (i.e. the curvature locus is a planar region), then there exists a unique umbilical focus at $p$ given by $$a=p+\frac{1}{\kappa_u}v_a^3.$$ In fact, here $\kappa_u=|\kappa_{v_a^3}|$. Similar results can be obtained when $\dim Aff_p<2$, although instead of one umbilical focus there may be a line of umbilical foci. The proof relies on the analysis of the Hessian of the distance squared function and is analogous to the proof of Proposition 6.6 in \cite{Deolindo/Oset}. 
\end{rem}


\begin{thebibliography}{22}

\bibitem{BastoGoncalves} {\sc J. Basto-Gon\c{c}alves} {\it Local geometry of surfaces in $\mathbb R^4$}.  Preprint (2013), arXiv:1304.2242.



\bibitem{BenediniOset2} {\sc P. Benedini Riul and R. Oset Sinha} {\it Relating second order geometry of manifolds through projections and normal sections}.  Publ. Mat. 65 (2021), no. 1, 389--407.

\bibitem{Benedini/Sinha/Ruas} {\sc P. Benedini Riul, R. Oset Sinha and M. A. S. Ruas} {\it The geometry of corank $1$ surfaces in $\mathbb{R}^{4}$}. Q. J. Math. 70 (2019), no. 3, 767--795.

\bibitem{Benedini/Sinha/Ruas2} {\sc P. Benedini Riul, R. Oset Sinha and M. A. S. Ruas} {\it Curvature loci of 3-manifolds}. Preprint (2022).

\bibitem{BenediniRuasSacramento} {\sc P. Benedini Riul, M. A. S. Ruas and A. de Jesus Sacramento} {\it Singular 3-manifolds in $\mathbb{R}^{5}$}.  Rev. R. Acad. Cienc. Exactas F\'is. Nat. Ser. A Mat. RACSAM 116 (2022), no. 1, Paper No. 56, 18 pp.

\bibitem{Binotto/Costa/Fuster} { \sc R. R. Binotto, S. I. Costa and M. C. Romero Fuster} {\it The curvature Veronese of a 3-manifold in Euclidean space}. Real and complex singularities: Amer. Math. Soc., Providence, RI (2016), (Contemp. Math., v. 675), p. 25--44.

\bibitem{bivianuno} {\sc C. Bivi\`{a}-Ausina and J. J. Nu\~{n}o-Ballesteros} {\it Multiplicity of iterated Jacobian extensions of weighted homogeneous map germs}.   Hokkaido Math. J. 29 (2000), no. 2, 341--368.





\bibitem{Deolindo/Oset} {\sc J. L. Deolindo-Silva and R. Oset Sinha} {\it Geometry of surfaces in $\mathbb R^5$ through projections and normal sections.}  Rev. R. Acad. Cienc. Exactas F\'is. Nat. Ser. A Mat. RACSAM 115 (2021), no. 2, Paper No. 81, 19 pp.



\bibitem{HHNUY} {\sc{M. Hasegawa, A. Honda, K. Naokawa, M. Umehara and K. Yamada}} {\it Intrinsic invariants of cross caps}. Selecta Math. (N.S.) 20 (2014), no. 3, 769--785.

\bibitem{Livro} {\sc{S. Izumiya, M. C. Romero Fuster, M. A. S. Ruas and F. Tari}} {\it Differential Geometry from Singularity Theory Viewpoint}. World Scientific Publishing Co. Pte. Ltd., Hackensack, NJ, 2016. xiii+368 pp. ISBN: 978-981-4590-44-0



\bibitem{Little} {\sc{J. A. Little}} {\it On singularities of submanifolds of higher dimensional Euclidean
spaces}. Ann. Mat. Pura Appl. 83 (4) (1969), 261--335.

\bibitem {MartinsBallesteros} {\sc{L. F. Martins and J. J. Nu\~{n}o-Ballesteros,}} {\it{Contact properties of surfaces in $\mathbb{R}^{3}$ with corank $1$ singularities}}. Tohoku Math. J. 67 (2015), 105--124.

\bibitem{MartinsSaji} {\sc{L. F. Martins and K. Saji,}} {\it Geometric invariants of cuspidal edges}. Canadian J. Math 68 (2016), no. 2, 445--462.


\bibitem{MochidaFusterRuas3}{\sc{D. K. H. Mochida, M. C. Romero Fuster and M. A. S. Ruas,}} {\it Inflection points and nonsingular embeddings of surfaces in $\mathbb R^5$}, Rocky Mountain J. Math. 33 (2003) 995--1010.

\bibitem{Montaldi}{\sc{J. A. Montaldi,}} {\it Contact with applications to submanifolds}, Ph.D. Thesis, University of Liverpool (1983).


\bibitem{BallesterosTari} {\sc{J. J. Nu\~{n}o-Ballesteros and F. Tari,}} {\it Surfaces in $\mathbb{R}^{4}$ and their projections to $3$-spaces}. Proc. Roy. Soc. Edinburgh Sect. A, 137 (2007), 1313--1328.

\bibitem{NunoRomero} {\sc{J. J. Nu\~{n}o-Ballesteros and M. C. Romero Fuster,}} {\it Contact properties of codimension 2 submanifolds with flat normal bundle}. Rev. Mat. Iberoam. 26 (2010), no. 3, 799--824.

\bibitem{Oset/Saji} {\sc{R. Oset Sinha and K. Saji,}} {\it On the geometry of folded cuspidal edges}.   Rev. Mat. Complut. 31 (2018), no. 3, 
627--650.

\bibitem{OsetSaji} {\sc{R. Oset Sinha and K. Saji,}} {\it The axial curvature for corank 1 singular surfaces}. To appear in Tohoku Math. J. arXiv:1911.08823  (2019).

\bibitem{OsetSinhaTari} {\sc{R. Oset Sinha and F. Tari,}} {\it Projections of surfaces in $\mathbb{R}^{4}$ to $\mathbb{R}^{3}$ and the geometry of their singular  images}. Rev. Mat. Iberoam. 32 (2015), no. 1, 33--50.

\bibitem{OsetSinhaTari2} {\sc{R. Oset Sinha and F. Tari,}} {\it On the flat geometry of the cuspidal edge}. Osaka J. Math. 55 (2018), no. 3, 393--421.

\bibitem{SUY} {\sc{K. Saji, M. Umehara, and K. Yamada,}} {\it The geometry of fronts}. Ann. of Math (2) 169 (2009), 491--529.

\bibitem{teramoto} {\sc{K. Teramoto,}} {\it Principal curvatures and parallel surfaces of wave fronts}.  Adv. Geom. 19 (2019), no. 4, 541--554.



\end{thebibliography}
\end{document}